\documentclass{article}

\usepackage{style}

\newcommand{\p}{\mathfrak{p}}
\newcommand{\rk}{\operatorname{rk}}

\newcommand{\catBun}{{\mathcal{B}un}}

\renewcommand{\Re}{\operatorname{Re}}

\renewcommand{\le}{\leqslant}
\renewcommand{\ge}{\geqslant}

\newcommand{\Ext}{\mathrm{Ext}}
\newcommand{\M}{\mathcal{M}}

\newcommand{\Mer}{\mathcal{M}er}
\newcommand{\SH}{\mathcal{SH}}
\newcommand{\Sh}{\mathrm{Sh}}

\newcommand{\shufflemult}{\, \circledS\, }
\newcommand{\CT}{\operatorname{CT}}
\newcommand{\Gr}{\operatorname{Gr}}
\newcommand{\Supp}{\operatorname{Supp}}
\newcommand{\Ch}{\operatorname{Ch}}

\title{The spherical Hall algebra of $\overline{\Spec(\mathcal{O}_K})$}
\author{Benjamin Li, Luis Modes\\Mentor: Haoshuo Fu\\Project suggested by: Zhiwei Yun}
\date{SPUR 2024}

\begin{document}

\setcounter{section}{0}

\setstretch{1}

\maketitle

\setcounter{section}{-1}

\begin{abstract}
    We generalize the result of M. Kapranov, O. Schiffmann, and E. Vasserot \cite{originalpaper} by showing that, for a number field $K$ with class number 1, the spherical Hall algebra of $\overline{\Spec(\mathcal{O}_K)}$, where $\OO_K$ is the ring of integers of $K$, is isomorphic to the Paley-Wiener shuffle algebra associated to a Hecke $L$-function corresponding to $K$.
\end{abstract}

\tableofcontents


\section{Introduction}

\blank M. Kapranov, O. Schiffmann, and E. Vasserot showed in \cite{originalpaper} that the spherical Hall algebra of $\overline{\Spec(\ZZ)}$ is isomorphic to the Paley-Wiener shuffle algebra associated to the global Harish-Chandra function $\Phi(s)$, defined as $\Phi(s)=\frac{\zeta^*(s)}{\zeta^*(s+1)}$, where $\zeta^*(s)=\pi^{-\frac{s}{2}}\Gamma(\frac{s}{2})\zeta(s)$ is the completed zeta function. 

In order to define the Hall algebra of $\overline{\Spec(\ZZ)}$, \textit{vector bundles on }$\overline{\Spec(\ZZ)}$ are defined as triples $E=(L,V,q)$ where $V=\RR^n$, $L\subset V$ is a $\ZZ$-lattice of maximal rank, and $q$ is a positive-definite quadratic form. The \textit{rank of }$E$ is defined as $\rk(E)=\dim_\RR(V)=\rk_\ZZ(L)$. These vector bundles form a category $\catBun$. The Hall algebra of $\overline{\Spec(\ZZ)}$ is defined as $H=\bigoplus_{n\geqslant 0}H_n$, where $H_n=C_c^\infty(\Bun_n)$ and $\Bun_n$ consists of the isomorphism classes of rank $n$ vector bundles. For this,  it is shown that $\Bun_n$ can be identified with $\GL_n(\ZZ)\doublecoset{\GL_n(\RR)}O_n$ and $\GL_n(\QQ)\doublecoset{\GL_n(\AA_\QQ)}O_n\times\prod_p\GL_n(\ZZ_p)$, and that it has a $C^\infty$-orbifold structure. A \textit{Hall product} $*$ is constructed and proven to make $H$ into a graded associative algebra with unit 1. The \textit{spherical Hall algebra} $SH$ is defined as the subalgebra generated by $H_1\subset H$. It is shown that we can identify $H_1$ with $C_c^\infty(\RR_+)$, so $SH$ can be seen as a subset of $\bigoplus_{n\geqslant0}C_c^\infty(\RR_+^n)$.

On the other hand, $\bigoplus_{n\geqslant 0}\Mer(\CC^n)$ can be made into a graded associative algebra with unit 1 via the \textit{shuffle product} $\shufflemult$, defined as
\[
F \shufflemult F' = \sum_{w \in \Sh(m, n)} w(F \ten F') \cdot \Phi_{w\inv}.
\]

Here, $\Sh(m,n)$ is the set of $(m,n)$-shuffles, and $\Phi_{w\inv}(s_1, \dots, s_{m+n}) = \prod_{i < j \text{ and } w(i) > w(j)} \Phi(s_{w(i)} - s_{w(j)})$. The \textit{Paley-Wiener shuffle algebra associated to} $\Phi$ is denoted by $\SH(\Phi)_\PW$ and is defined as the subalgebra generated by the subspace $\PW(\CC)\subset\bigoplus_{n\geqslant 0}\Mer(\CC^n)$ of Paley-Wiener functions.

The main result of \cite{originalpaper} shows that $SH\simeq\SH(\Phi)_\PW$. 

\begin{theorem}
The Mellin transform $\M:SH_1=C_c^\infty(\RR_+)\overset{\sim}{\longrightarrow}\PW(\CC)$ extends to an isomorphism of algebras $SH\rightarrow\SH(\Phi)_\PW$. 
\end{theorem}

\noindent The explicit isomorphism $\Ch:SH\overset{\sim}{\longrightarrow}\SH(\Phi)_\PW$ is defined as $\Ch=\M\circ\~\CT$, where $\M$ is the \textit{Mellin transform} $\M:C^\infty(\RR_+^n)\to \PW(\CC^n)$ given by $\M(f)(s)=\int_{a\in\RR^n_+}f(a)a^s\frac{\diff a}{a}$, and $\~\CT:C_c^\infty(\RR_+^n)\to C^\infty(\RR_+^n)$ is a twisted version of the \textit{constant term map} $\CT:H_n=C_c^\infty(\GL_n(\ZZ)\doublecoset{\GL_n(\RR)}O_n)\to C^\infty(\RR_+^n)$, which takes an automorphic form on $\GL_n$ and yields its constant term.

Note that, in particular, $\ZZ$ is the ring of integers of the field $\QQ$. In our paper, we generalize this result for a number field $K$ with class number 1, that is, such that its ring of integers $\OO_K$ is a principal ideal domain. We follow \cite{originalpaper} closely.

We consider the real and complex Archimedean places $\nu\in\s$ of $K$, which correspond to embeddings $\nu:K\to \CC$.
From this, we generalize $\ZZ\subset\RR$ to $\OO_K\subset R=\prod_{\nu\in\s}R_\nu$, where $R_\nu=\RR$ if $\nu\in\s_\RR$ and $R_\nu=\CC$ if $\nu\in\s_\CC$, and the embedding is given by $a\mapsto(\nu(a))_{\nu\in\s}$.
In our definition of vector bundles over $\overline{\Spec(\OO_K)}$, $E=(L,V,q)$, where $V=R^n$, $L\subset R^n$ is an $\OO_K$-lattice of maximal rank, and $q=\prod_{\nu}q_\nu$, where $q_\nu$ is a positive-definite quadratic form if $\nu\in\s_\RR$ and a positive-definite Hermitian form if $\nu\in\s_\CC$.
From this, we can define $\Bun_n$, which will now be identified with $$\GL_n(\OO_K)\doublecoset{\GL_n(R)} \K_n= \GL_n(K)\doublecoset{\GL_n(\A_K)} \widehat{\K_n},$$where $\K_n=\prod_{\nu\in\s}O_\nu$ and $\widehat{\K_n}=\prod_{\nu\in\s}O_\nu\times\prod_\p \GL_n(\OO_{K, \p})$.
Here, $O_\nu$ is the orthogonal group for $\nu\in\s_\RR$ and the unitary group for $\nu\in\s_\CC$, and $\p$ goes over the prime ideals of $\OO_K$ with $\OO_{K, \p}$ being the $\p$-adic integers.
In the same  way as before, we define the Hall algebra of $\overline{\Spec(\OO_K)}$, and the spherical Hall algebra $SH$ is again defined as the subalgebra generated by $H_1\subset H$. However, this time we will be able to identify $SH$ as a subset of $\bigoplus_{n\geqslant0}C_c^\infty(\B^n)$, where $\B = \Bun_1 = \OO_K^\times\backslash\prod_{\nu\in\s}\RR_+^\times \simeq \OO_K^\times\backslash\prod_{\nu\in\s} \RR$ via taking logarithms.
We know (say, from \cite{milne}) that $\Lambda = \OO_K^\times$ is (non-canonically) isomorphic to $\ZZ^{r_1+r_2-1}$, where $r_1 = \#\s_\RR$ and $r_2 = \#\s_\CC$ are the number of real and complex places, respectively. 
From this, $$\B=\OO_K^\times\backslash\prod_{\nu\in\s}\RR_+^\times\simeq(\RR^{r_1+r_2-1}/\Lambda)\times\RR_+ \simeq (\RR/\ZZ)^{r_1+r_2-1}\times\RR_+.$$As a result, our Fourier transform will now have the form $\F:C_c^\infty(\B^n)\to \PW((\Lambda^*\times\CC)^n)$, where $$\F f(\lambda^*,s) = \int_{x \in D} f(x) \abs{x}_\B^s x^{2 \pi i \lambda^*} \frac{\diff x}{\abs{x}_\B}.$$Here, $|x|_\B=\prod_{i=1}^{r_1}x_i\prod_{j=r_1+1}^{r_1+r_2}x_j^2$ and $D$ is a fundamental domain for $\B^n$. We modify the constant term map $\~\CT:C_c^\infty(\B^n)\to C^\infty(\B^n)$ and the homomorphism $\Ch=\F\circ\~\CT$ accordingly. Instead of getting $\zeta(s)$ in our computations, we end up getting a Hecke $L$-function corresponding to $K$ (for the definition and properties of this function, see \cite{neukirch}), which we call $L_K(\lambda^*,s)$, and we define $\Phi_K(\lambda^*,s)$ accordingly. After making $\bigoplus_{n\geqslant 0}\Mer((\Lambda^*\times\CC)^n)$ into an algebra via $\shufflemult$ and defining the \textit{Paley-Wiener shuffle algebra with respect to }$\Phi_K(\lambda^*,s)$, which we denote by $\SH(\Phi_K)_\PW$, as the subalgebra generated by $\PW(\Lambda^*\times\CC)$, we prove the following main result. 

\begin{theorem} The map $\Ch:SH\rightarrow\SH(\Phi_K)_\PW$ is an isomorphism of algebras. 
\end{theorem}

\bigskip

\noindent\textbf{Structure of the paper}

In section \ref{sec: vector bundles}, we define vector bundles on $\overline{\Spec(\OO_K)}$, morphisms between them, their rank, and their degree. We prove that we have bijections $$\Bun_n=\GL_n(\OO_K)\doublecoset{\GL_n(R)} \K_n= \GL_n(K)\doublecoset{\GL_n(\A_K)} \widehat{\K_n}$$Afterwards, we define \textit{subbundles} $E'\subset E$, which we will use for our definition of Hall product. In section \ref{sec: Hall algebra}, we define the Hall product $*:H_m\otimes_\CC H_n \to H_{m+n}$ as $$(f*g)(E)=\sum_{\substack{E'\subset E\\\rk(E')=m}}\deg(E')^{\frac{n}{2}}\deg(E/E')^{-\frac{m}{2}}\cdot f(E')g(E/E')$$and we show that it makes $H=\bigoplus_{n=0}^\infty H_n$ into a graded associative algebra with unit 1. We define the spherical Hall algebra $SH$. 

In section \ref{sec: L-function}, we define the set $\Mer((\Lambda^* \times \CC)^n)$ of meromorphic functions on $(\Lambda^* \times \CC)^n$, and we define $L_K(\lambda^*,s)$, a \textit{Hecke} $L$\textit{-function corresponding to }$K$.
Then, we define $\Phi_K(s,\lambda^*)$ as $\Phi_K(\lambda^*,s)=\frac{L_K^*(\lambda^*,s)}{L_K^*(\lambda^*,s+1)}$, where $L_K^*(\lambda^*,s)=\pi^{-(r_1/2)s} (2\pi)^{-r_2 s}\Gamma_K(\lambda^*, s)L_K(\lambda^*,s)$.  In section \ref{sec: shuffle algebra}, using $\Phi_K(\lambda^*,s)$, we define the shuffle multiplication $\shufflemult$, which makes $\bigoplus_{n\geqslant0}\Mer((\Lambda^*\times\CC)^n)$ into a graded associative algebra with unit 1. We then define the subalgebra $\SH(\Phi_K)_\PW$.

In section \ref{sec: Fourier transform}, we define Fourier transform $\F$ and its inverse transform $\G$. In section \ref{sec: intertwiner}, we define $M_w:C_c^\infty(\B^n)\to C^\infty(\B^n)$, the principal series intertwiner, as $$(M_w\vp)(g)=\int_{u\in U_w(\AA_K)\backslash U(\AA_K)}\varphi(wug)\diff u.$$Here, $w$ is a permutation and $U_w=U\cap(w^{-1}Uw)$. This operator will allow us to connect the Hall product, the Fourier transform, and the constant term, which we will define next.

In section \ref{sec: isomorphism}, we introduce the \textit{constant term} $\CT_n:H_n\to C^\infty(\B^n)$ as $$\CT_n(a_1,\dots,a_n)=\int_{u\in U_n(\OO_K)\backslash U_n(R)}f(u\cdot \diag(a_1,\dots,a_n))\diff u=\int_{u_\AA\in U_n(K)\backslash U_n(\AA_K)}f(u_\AA\cdot \diag(a_1,\dots,a_n))\diff u_\AA$$and its twist $\~\CT$ as $\~\CT(f)(a)=\CT(f)(a)\cdot \delta(a)^{\frac{1}{2}}$, where $\delta(a)$ is the Iwasawa Jacobian. We will use these to construct the map $\Ch_n=\F\circ\~\CT_n$. We will show that $\Ch:SH\to \SH(\Phi_K)_\PW$ is an isomorphism. 

\bigskip

\noindent\textbf{Notation}

We use $\AA_K$ to denote the adelic ring of the number field $K$ and $\N$ to denote the ideal norm. We use $z$ to denote a tuple $((\lambda^*_1, s_1), \ldots,(\lambda^*_n, s_n))\in(\Lambda^*\times\CC)^n$. To denote the set $\{1,\dots,n\}$ by $\db{1,n}$, and we use $\# S$ to denote the cardinality of the set $S$. We use $p^\ZZ$ to denote the set of integer powers of $p$. We write $A_n$ for the subgroup of diagonal matrices in $\GL_n$.

\bigskip

\noindent\textbf{Acknowledgements}

Benjamin Li and Luis Modes would like to thank their mentor, Haoshuo Fu, for his patient and continuous support during this research and at the time of writing this paper. We are also indebted to Zhiwei Yun, who proposed this research project, and to David Jerison and Jonathan Bloom, who provided helpful insights and suggestions. All authors were funded by the MIT Department of Mathematics
through its Summer Program in Undergraduate Research (SPUR), to which we are also deeply grateful.


\section{Vector bundles on $\overline{\Spec(\mathcal{O}_K)}$}
\label{sec: vector bundles}

\blank Let $K$ be a number field such that its ring of integers $\OO_K$ is a principal ideal domain. In other words, the class number of $K$ is 1. Let $\s$ be the set of real and complex Archimedean places of $K$. Recall that a place $\nu\in\s$ corresponds to an embedding $\nu:K\to \CC$ and a norm $|\cdot|_\nu:\nu(K)\to\RR_+$. Let $\s_\RR$ and $\s_\CC$ denote the sets of real and complex places, respectively, and let $r_1 = \#\s_\RR$ and $r_2 = \#\s_\CC$. Let $R=\prod_{\nu\in\s}R_\nu$, where $R_\nu=\RR$ for $\nu\in\s_\RR$ and $R_\nu=\CC$ for $\nu\in\s_\CC$.

A \textit{vector bundle of rank }$n$\textit{ over }$\overline{\Spec(\OO_K)}$ is a triple $E=(L,V,q)$, where $V=R^n$, $L\subset V$ is an $\OO_K$-lattice of maximal rank (i.e., a free $\OO_K$-module of rank $n$  such the map $L\otimes_{\OO_K}R\rightarrow V$ induced by the inclusion $L\subset V$ is an isomorphism of $R$-modules), and $q=\prod_{\nu\in\s}q_\nu$, where the $q_\nu$ is positive-definite quadratic form on $\RR^n$ for $\nu\in\s_\RR$ and a positive-definite Hermitian form on $\CC^n$ for $\nu\in\s_\CC$. The \textit{rank} of a vector bundle $E=(L,V,q)$ is defined as $\rk(E)=\rk_RV=\rk_{\OO_K}L$. For example, if $K=\QQ(i)$, we have the \textit{trivial bundle} $(\ZZ[i],\CC,z\overline{z})$, and if $K=\QQ(\sqrt{2})$, we have a vector bundle $(\{(a+b\sqrt{2},a-b\sqrt{2}):a,b\in\ZZ\},\RR\times\RR,x^2\times3x^2)$. Both of these examples have rank 1.

A \textit{morphism} of vector bundles $f:E=(L,V,q)\to E'=(L',V',q')$ is an $R$-linear map $f:V\rightarrow V'$ such that $f(L)\subset L'$ and $q'_\nu(f(v))\le q_\nu(v)$ for all $v\in V$ and $\nu\in\s$. We say that $f:E \to E'$ is an \textit{isomorphism} if  $f:V \to V'$ is an $R$-isomorphism, $f(L)=L'$, and $q'(f(v))=q(v)$ for all $v\in V$. As a result, we get a category $\catBun$ where the objects are vector bundles and the morphisms are as previously defined.

\subsection{Isomorphism classes of rank $n$ vector bundles}

\blank Let us consider the class of isomorphism classes of rank $n$ vector bundles, which we denote by $\Bun_n$.
To be able to define functions in $\Bun_n$, we show that we can identify $\Bun_n$ in the following two ways:
\[
\Bun_n\simeq\GL_n(\OO_K)\backslash \GL_n(R)/\K_n\simeq\GL_n(K)\backslash \GL_n(\A_K)/\widehat{\K_n}.
\]

Here, $\K_n=\prod_{\nu\in\s}O_\nu$, where $O_\nu$ is the orthogonal group $\mathrm{O}_n$ for $\nu\in\s_\RR$ and the unitary group $\mathrm{U}_n$ for $\nu\in\s_\CC$, and $\widehat{\K_n}=\prod_{\nu\in\s}O_\nu\times\prod_{\p \in \Spec(\OO_K)}\GL_n(\OO_{K, \p})$, where $\OO_{K, \p}$ denotes the set of $\pp$-adic integers. The inclusion $\GL_n(\OO_K)\subset\GL_n(R)$ is induced from the embedding $\OO_K\subset R$ given by $a\mapsto(\nu(a))_{\nu\in\s}$, and the other inclusions come from the natural embeddings as well.

Before proving this statement, let us first exhibit some lemmas.

\begin{lemma} \label{lem: union}
    For $(p) = \pp \in \Spec(\OO_K)$, we have
    \[
    \GL_n(K_\p) = \GL_n(\OO_K) A_n(p^\ZZ) \GL_n(\OO_{K, \p}).
    \]
\end{lemma}

\begin{proof}
    Pick any $g \in \GL_n(K_\p)$.
    We split into several steps.

    \begin{enumerate}
        \item After left and right multiply by matrices in $\GL_n(\ZZ)$ we can assume $g_{11}$ has the smallest $v_\p$ value.
        \item After right multiply by a matrix in $\GL_n(\OO_{K, \p})$ we can assume $g_{1j} = 0$ for $j > 1$.
    Note that $g_{11}$ still has the smallest $v_\p$ value.
        \item Repeat Step 1 and 2 to the $(n - 1) \times (n - 1)$ submatrix (deleting the first row and column) we can assume $g_{2j} = 0$ for $j > 2$; note that $g_{11}$ still has the smallest $v_\p$ value, while $g_{22}$ has the smallest $v_\p$ value in the submatrix.
        \item Repeat Step 3 to submatrices we can assume $g$ is lower-triangular, with $g_{ii}$ having the smallest $v_\p$ value in the submatrix $(g_{\ell j})_{\ell, j \ge i}$.
        \item After right multiply by a matrix in $\GL_n(\OO_{K, \p})$ we can assume $g_{ii} = p^{a_i}$ for some $a_1 < a_2 < \cdots < a_n$.
        \item After left multiply by a matrix in $\GL_n(\OO_K)$ we can assume $v_\p(g_{i1}) \ge a_n$.
        \item Repeat Step 6 to the submatrices we can assume $v_\p(g_{ij}) \ge a_n$ for $i > j$.
        \item  Finally, after right multiply a matrix in $\GL_n(\OO_{K, \p})$ we can assume $g = \diag\pr{p^{a_1}, \ldots, p^{a_n}}$, which finishes the proof.
    \end{enumerate}
\end{proof}

\begin{cor} \label{cor: finite adeles}
    We have
    \[
    \GL_n\pr{\A_K^f} = \GL_n(K) \GL_n\pr{\widehat{\OO}_K},
    \]
    where $\A_K^f$ is the ring of finite adeles in $K$ and $\widehat{\OO}_K$ is the profinite completion of $\OO_K$, so that $\GL_n(\widehat{\OO}_K) = \prod_{\p \in \Spec(\OO_K)} \GL_n(\OO_{K, \pp})$.
\end{cor}

\begin{proof}
    Pick any $g = (g_\p)_{\p \in \Spec(\OO_K)} \in \GL_n(\A_K^f)$, where $g_\p \in \GL_n(\OO_{K, \p})$ for all $\p \in \Spec(\OO_K)- \br{\p_1, \ldots, \p_m}$ being a finite set.
    Let $(p_i) = \p_i$.
    We then do the following procedure.

    \begin{enumerate}
        \item From Lemma~\ref{lem: union} we can find $u_1 \in \GL_n(\OO_K[1/p_1])$ such that $u_1 g_{\p_1} \in \GL_n(\OO_{K, \p_1})$.
        \item  From Lemma~\ref{lem: union} we can find $u_2 \in \GL_n(\OO_K[1/p_2])$ such that $u_2 u_1 g_{\p_2} \in \GL_n(\OO_{K, \p_2})$.
        \item[\vdots]
        \item[$m.$] From Lemma~\ref{lem: union} we can find $u_m \in \GL_n(\OO_K[1/p_m])$ such that $u_m u_{m - 1} \cdots u_1 g_{\p_m} \in \GL_n(\OO_{K, \p_m})$.
        
    \end{enumerate}

     Now put $u = u_m u_{m - 1} \cdots u_1 \in \GL_n(K)$, then we have $u g_{\p_i} \in \GL_n(\OO_{K, \p_i})$ for $1 \le i \le m$ because $\OO_K[1/p_i] \sub \OO_{K, \p_j}$ for $i \ne j$.
    For the same reason we have $u g_\p \in \GL_n(\OO_{K, \p})$ for all $p \in \Spec(\OO_K)- \br{\p_1, \ldots, \p_m}$.
    Thus, we have $g = u\inv \pr{u g_\p}_{\p \in \Spec(\OO_K)} \in \GL_n(K) \GL_n(\widehat{\OO}_K)$, which finishes the proof.
\end{proof}

Now we are at a position to prove the two identifications of $\Bun_n$.

\begin{prop}\label{prop: Bun and double cosets}
We have bijections
\[
\Bun_n=\GL_n(\OO_K)\backslash \GL_n(R)/\K_n=GL_n(K)\backslash GL_n(\A_K)/\widehat{\K_n}.
\]
\end{prop}

\begin{proof}First, let us show that $\Bun_n=\GL_n(\OO_K)\backslash \GL_n(R)/\K_n$. Let $(M_*q)(x)=q(Mx)$ for any matrix $M$ and any quadratic form $q$. If $q_{\text{st}}(x)=x^Tx$, $q(x)=x^TA^TAx=A_*q_{\text{st}}$ (as any symmetric matrix may be written in the form $A^TA$), and $B\in GL_n(R)$ is a matrix such that $B(L)=\OO_K$.  Consider the maps $\alpha:\GL_n(\OO_K)\backslash \GL_n(R)/\K_n\to\Bun_n$ and $\beta:\Bun_n\to\GL_n(\OO_K)\backslash \GL_n(R)/\K_n$ defined by $$\alpha:g\mapsto (\OO_K^n,R^n,(g^{-1})_*(q_{\text{st}}))$$and $$\beta:(L,R^n,q)\overset{B}{\simeq}(\OO_K^n,R^n,(AB^{-1})_*(q_{\text{st}}))\mapsto BA^{-1}$$To check that $\alpha$ is well-defined, if we choose the representative $hgk$ instead, note that $$(\OO_K^n,R^n,((hgk)^{-1})_*(q_{\text{st}}))\overset{ h^{-1}}{\simeq}(\OO_K^n,R^n,((gk)^{-1})_*(q_{\text{st}}))=(\OO_K^n,R^n,(g^{-1})_*(q_{\text{st}}))$$To see that $\beta$ is well-defined, if $(L,R^n,q)\overset{f}{\simeq}(L',R^n,q')$, we have that $B'f=B$ and $f_*q'=q\Longrightarrow A'f=A$ (both equalities in the double coset), so it follows that $BA^{-1}=B'A'^{-1}$ in the double coset. 

Finally, to check these are inverses, note that $$(\alpha\circ\beta)(L,R^n,q)=\alpha(BA^{-1})=(\OO_K^n,R^n,(AB^{-1})_*(q_{\text{st}}))\simeq (L,R^n,q)$$and $$(\beta\circ\alpha)(g)=\beta(\OO_k^n,R^n,(q^{-1})_*(q_{\text{st}}))=g$$as desired. Now, let us show that $\GL_n(\OO_K)\backslash \GL_n(R)/\K_n=GL_n(K)\backslash GL_n(\A_K)/\widehat{\K_n}$.

Now we will show that the natural map
\[
\GL_n(\OO_K)\backslash \GL_n(R)/\K_n \to GL_n(K)\backslash GL_n(\A_K)/\widehat{\K_n}
\]
induced from the inclusion $\GL_n(R) \hookrightarrow GL_n(\A_K)$ is a bijection.
It is surjective from Corollary~\ref{cor: finite adeles}.
Now we show the injectivity.

Let us say that $\ob{\pr{g_\nu}_{\nu \in \s}}$ has the same image as $\ob{\pr{g_\nu'}_{\nu \in \s}}$.
Then there exists $u \in \GL_n(K)$ and $\pr{(h_\p)_{\p \in \Spec(\OO_K)}, (h_\nu)_{\nu \in \s}} \in K_n = \prod_{\p \in \Spec(\OO_K)} \GL_n(\OO_{K, \p}) \times \prod_{\nu \in \s} O_\nu$ such that $\pr{(\id)_{\p \in \PP}, \pr{g_\nu'}_{\nu \in \s}} = u\inv \pr{(\id)_{\p \in \Spec(\OO_K)}, \pr{g_\nu}_{\nu \in \s}} \pr{(h_\p)_{\p \in \Spec(\OO_K)}, (h_\nu)_{\nu \in \s}}$, that is, $h_\p = u$ and $g_\nu' = u\inv g_\nu h_\nu$.

However, then $u \in \GL_n(K) \cap \bigcap_{\p \in \Spec(\OO_K)} \GL_n(\OO_{K, \p}) = \GL_n(\OO_K)$ since $\OO_K = K \cap \bigcap_{\p \in \Spec(\OO_K)} \OO_{K, \p}$ and $\OO_K^\times = K^\times \cap \bigcap_{\p \in \Spec(\OO_K)} \OO_{K, \p}^\times$, and so $\ob{\pr{g_\nu}_{\nu \in \s}} = \ob{\pr{g_\nu'}_{\nu \in \s}}$, which concludes the proof.
\end{proof}

\noindent From now on, when we say a vector bundle $E\in\Bun_n$, we will actually mean the isomorphism class of $E$.

Now, let us compute $\Bun_1$ explicitly. As the orthogonal and unitary groups are just $\{\pm1\}$ and $S_1$ for $n=1$, respectively, we have that
\begin{align*}
\Bun_1&=\GL_1(\OO_K)\backslash \GL_1(R)/\prod_{\nu\in\s}(O_\nu)_1\\
    &=\OO_K^\times\backslash R^\times/\prod_{\nu\in\s_\RR} \{\pm1\}\times\prod_{\nu'\in\s_\CC}  S_1\\
    &=\OO_K^\times\backslash \prod_{\nu\in\s_\RR} \RR^\times\times\prod_{\nu'\in\s_\CC}\CC^\times/\prod_{\nu\in\s_\RR} \{\pm1\}\times\prod_{\nu'\in\s_\CC}  S_1\\
    &=\OO_K^\times\backslash \prod_{\nu\in\s_\RR} \RR^\times/\{\pm1\}\times\prod_{\nu'\in\s_\CC}\CC^\times/S_1\\
    &=\OO_K^\times\backslash \prod_{\nu\in\s}\RR_+^\times.
\end{align*}

Here, $\OO_K^\times$ embeds into $\prod_{\nu\in\s}\RR_+^\times$ by $a\mapsto(|\nu(a)|_\nu)_{\nu\in\s}$.  Furthermore, by \cite{milne}, $\OO_K^\times$ is a $\ZZ$-lattice of rank $r_1+r_2-1$, so if we denote this lattice by $\Lambda$, we have that
\[
\Bun_1=\OO_K^\times\backslash \prod_{\nu\in\s}\RR_+^\times = (\RR^{r_1 + r_2 - 1} / \Lambda) \times \RR_+^\times \simeq (\RR / \ZZ)^{r_1 + r_2 - 1} \times \RR_+^\times
\]
where we used the isomorphism $\log:\RR_+^\times \overset{\sim}{\longrightarrow} \RR$.
However, the isomorphism $\RR^{r_1 + r_2 - 1} / \Lambda \simeq (\RR / \ZZ)^{r_1 + r_2 - 1}$ is not canonical, so we prefer to use $\RR^{r_1 + r_2 - 1} / \Lambda$ instead of $(\RR / \ZZ)^{r_1 + r_2 - 1}$.
This identification of $\Bun_1$ will allow us to define functions  on it, as well as a Fourier transform for these functions.
Denote $\B = \Bun_1 = \OO_K^\times \backslash \prod_{\nu \in \s} \RR_+^\times = (\RR^{r_1 + r_2 - 1} / \Lambda) \times \RR_+^\times$ for the rest of the paper.
We can write an element $a \in \B$ as $\ob{(a_i)}^{r_1+r_2}_{i=1}$, where $a_i\in\RR_+^\times$ and we take the equivalence class after taking the quotient by $\OO_K$. We define a norm $|\cdot|_\B:\B \to \RR_+$ by 
\[
|a|_\B = \prod_{i = 1}^{r_1} a_i \prod_{j = r_1 + 1}^{r_1 + r_2} a_j^2.
\]
The complex norms are squared because we have two embeddings $\nu$ and $\overline{\nu}$ for each complex place $\nu$. Note that this norm is well-defined, as for $b \in \OO_K^\times$, we have that $|b|_\B = \abs{\pr{\abs{\nu(b)}}_{\nu \in \s}}_\B = \prod_{\nu\in\s_\RR}|\nu(b)|_\nu\prod_{\nu\in\s_\CC}|\nu(b)|^2_\nu=1$. 

Let $\beta$ be the map $\beta:\Bun_n \to \GL_n(\OO_K)\doublecoset{\GL_n(R)}\K_n$. We define the \textit{degree of} $E\in\Bun_n$ as $\deg(E)=|\beta(E)|_\B$. Note that we can take the norm of $\beta(E)\in R^\times$ because we have a natural map $$R^\times\to \OO_K^\times\backslash R^\times/\prod_{\nu\in\s_\RR} \{\pm1\}\times\prod_{\nu'\in\s_\CC}  S_1=\B.$$ We can check that $\deg$ is well-defined as $\det(h),\det(k)\in\OO_K$ for $h\in\GL_n(\OO_K)$ and $k\in \K_n$. 




\subsection{Subbundles}

We say that 

$$0 \longrightarrow E' \overset{i}{\longrightarrow} E \overset{j}{\longrightarrow} E'' \longrightarrow 0$$is a \textit{short exact sequence in }$\catBun$ if we have the following:

\begin{enumerate}

    \item The induced sequences of modules are short exact.
    
    \item The form $q'$ is equal to $i^*(q)$, where $(i^*q)(v')=q(i(v'))$, $v'\in V'$.

    \item The form $q''$ is equal to $j_*(q)$, where $(j_*q)(v'')=\min_{j(v)=v''}q(v)$, $v''\in V''$.
\end{enumerate}We say that $i$ is an \textit{admissible monomorphism}. We define a \textit{subbundle in }$E$ as the equivalence class of admissible monomorphisms $E'\to E$ modulo isomorphisms of the source. For each subbundle $E'$, we then have a quotient bundle $E/E'=E''$.

We have the following three propositions.

\begin{prop}\label{prop: subbundle lore}
Let $E=(L,V,q)$ be a vector bundle on $\overline{\Spec(\OO_K)}$. We have a bijection between each of the following sets:

\begin{enumerate}
    \item Rank $r$ subbundles $E'\subset E$.
    \item Rank $r$ primitive sublattices, that is, submodules $L'\subset L$ such that $L/L'$ has no torsion.
    \item $K$-linear subspaces $W'\subset L\otimes_{\OO_K} K$ of dimension $r$.
\end{enumerate}
\end{prop}

\begin{proof}
Given a rank $r$ subbundle $E'=(L',V',q')$, we can map it to a lattice $L'$. This lattice will be primitive because there must be some $E''$ such that $0\to L'\to L\to L''\to 0$ is exact, so $L/L'=L''$ would have to be free. This gives a bijection between $(1)$ and $(2)$. For a bijection between $(1)$ and $(3)$, we can map $E'$ to $L'\otimes_{\OO_K}K$. We can check that these maps are indeed bijections, as desired.
\end{proof}

\begin{prop}\label{prop: bounded degree}
Let $E=(L,V,q)$ be a vector bundle on $\overline{\Spec(\OO_K)}$. For any $r\in \ZZ_+$ and $a\in \RR_+$, the set of subbundles $E'\subset E$ with $\rk(E')=r$ and $\deg(E)\geqslant a$ is finite.
\end{prop}

\begin{proof}
The proof is the same as in \cite{originalpaper}.
\end{proof}

\begin{prop}\label{prop: torus}
The set $\Ext^1(E'',E')$ has a natural structure isomorphic to  $(R/\OO_K)^{n'n''}$, where $n'=\rk(E')$ and $n''=\rk(E'')$.
\end{prop}

\begin{proof}
The proof is the same as in \cite{originalpaper}.
\end{proof}


\section{Hall algebra}
\label{sec: Hall algebra}

\blank To be able to define functions in $\Bun_n$, we show the following.

\begin{prop}
The set $\Bun_n$ can be endowed with an orbifold structure.
\end{prop}

\begin{proof}
The proof is very similar to that in \cite{originalpaper}.
\end{proof}As a result, we can consider the set $H_n=C_c^\infty(\Bun_n)$ of compactly supported functions in $\Bun_n$. Let $H=\bigoplus_{n\geqslant0}H_n$ be the \textit{Hall algebra of }$\overline{\Spec(\OO_K)}$. However, it is still not clear that $H$ has an algebra structure. Our next step is to show this.

Given $f\in H_m$ and $g\in H_n$, we define the \textit{Hall product} $f*g\in H_{m+n}$ as

$$(f*g)(E)=\sum_{\substack{E'\subset E\\\rk(E')=m}}\deg(E')^{\frac{n}{2}}\deg(E/E')^{-\frac{m}{2}}\cdot f(E')g(E/E')$$

\bigskip

\begin{prop}\label{prop: Hall product}
$f*g$ is well-defined and makes $H$ into a graded associative algebra.    
\end{prop}

\begin{proof}
The proof is very similar to that in \cite{originalpaper}.
\end{proof}Now, let $SH$ be the \textit{spherical Hall algebra of }$\overline{\Spec(\OO_K)}$, which we define as the subalgebra $SH\subset H$ generated by $H_1$. In other words, $SH=\bigoplus_{n\geqslant0}H_1^{\otimes n}$. We define the map $*_{1^n}:H_1^{\otimes n}\to H_n$ as the \textit{multiplication map}, that is, $$\vp_1\otimes\dots\otimes\vp_n\mapsto\vp_1*\dots*\vp_n.$$

Translating the Hall product into group-theoretic terms, we have the following proposition.

\begin{prop}\label{prop: group version}
$f'*f''$ can be rewritten as 

$$(f'*f'')(g)=\sum_{\gamma\in P_{n',n''}(K)\backslash\GL_n(K)}f(\gamma g)$$where $P_{n',n''}$ is the parabolic subgroup of $\GL_n$, that is, the group of block lower triangular matrices, $g=(g',g'')$ and $f(\gamma g)=|\det(\gamma g')|^{\frac{n''}{2}}\cdot |\det(\gamma g'')|^{-\frac{n'}{2}}\cdot f'(\gamma g')f''(\gamma g'')$. This makes sense, as we have an isomorphism $$(\GL_n{n'}(K)\doublecoset{\GL_{n'}(\AA_K)}\widehat{\K_{n'}})\times(\GL_{n''}(K)\doublecoset{\GL_{n''}(\AA_K)}\widehat{\K_{n''}})\overset{\sim}{\longrightarrow}(U_{n',n''}(\AA_K)A_{n',n''}(K))\doublecoset{\GL_n(\AA_K)}\widehat{\K_n}$$coming from the Iwasawa decomposition.
\end{prop}


\section{Hecke $L$-function associated to $K$}
\label{sec: L-function}

\blank Let $H$ be the hyperplane in $\RR^{r_1 + r_2} \simeq \RR_+^{r_1 + r_2}$ that is in the kernel of $\abs{\cdot}$, and let $\Lambda$ be the lattice $\OO_K^\times$ in $H$.
Define dual lattice $\Lambda^*= \br{\lambda^* \in H^* : \brk{\lambda^*, \lambda} \in \ZZ \text{ for all } \lambda \in \Lambda}$, which is isomorphic to $\ZZ^{r_1 + r_2 - 1}$ as $\ZZ$-modules.
Note that for we can define group homomorphism $\brk{\lambda^*, \cdot} : \RR^{r_1 + r_2} \to \RR$ to be the composition of the projection $\RR^{r_1 + r_2} \to H$ and $\brk{\lambda^*, \cdot} : H \to \RR$.

For $\lambda^* \in \Lambda^*, s \in \CC$ such that $\Re s > 1$, we define
\[
    L_K(\lambda^*, s) = \sum_{(x) \sub \OO_K} \frac{1}{\exp(2 \pi i \brk{\lambda^*, \log \abs{x}}) \N(x)^s},
\]
where the summation runs through all ideals in $\OO_K$ and $\N$ denote the norm in $K$.
Here by $\log \abs{x}$ we mean the tuple $\pr{\log \abs{\nu(x)}}_{\nu \in \s}$ in $\RR^{r_1 + r_2}$.

The Euler product of $L_K$ reads
\[
L_K(\lambda^*, s) = \prod_{\pp \sub \OO_K} \pr{1 - \exp(-2 \pi i \brk{\lambda^*, \log \abs{p}}) \N(p)^{-s}}\inv,
\]
where the summation runs through all prime ideals in $\OO_K$ and $p$ is (any) generator for $\pp$.

We use $\Mer((\Lambda^*)^n \times \CC^n)$ to denote the set of meromorphic functions whose domain is $(\Lambda^*)^n \times \CC^n$, that is, those functions $f : (\Lambda^*)^n \times \CC^n \to \CC$ such that for each $\lambda^* \in (\Lambda^*)^n$, the function $f(\lambda^*, \cdot) : \CC^n \to \CC$ is a meromorphic function.
Also, we use $\OO((\Lambda^*)^n \times \CC^n)$ to denote the set of entire functions whose domain is $(\Lambda^*)^n \times \CC^n$.

Now we let
\begin{align*}
    \Gamma_K(\lambda^*, s) & = \prod_{\nu \in \s_\RR} \Gamma(s / 2 + \pi i \lambda^*_\nu) \prod_{\nu \in \s_\CC} \Gamma(s + \pi i \lambda^*_\nu) \\
    L_K^*(\lambda^*, s) & = \pi^{-(r_1 / 2)s} (2 \pi)^{-r_2 s} \Gamma_K(\lambda^*, s) L_K(\lambda^*, s) \\
    \Phi_K(\lambda^*, s) & = \frac{L_K^*(\lambda^*, s)}{L_K^*(\lambda^*, s + 1)}.
\end{align*}

From \cite{neukirch} we see we can extend $L_K^*$ to a meromorphic function in $\Lambda^* \times \CC$, which satisfies the functional equation $L_K^*(-\lambda^*, 1 - s) = L_K^*(\lambda^*, s)$.
Thus, we have
\begin{align*}
    \Phi_K(-\lambda^*, -s) & = \frac{L_K^*(-\lambda^*, -s)}{L_K^*(-\lambda^*, -s + 1)} \\
        & = \frac{L_K^*(\lambda^*, s + 1)}{L_K^*(\lambda^*, s)} \\
        & = \Phi_K(\lambda^*, s)\inv.
\end{align*}

Finally, a \textit{Paley-Wiener function} in $\Lambda^* \times \CC$ is defined as follows.
For any $\lambda^* \in \Lambda^*$ there exists a unique $\lambda^\vee \in H$ such that $\brk{\lambda^*, x} = \lambda^\vee \cdot x$ for any $x \in H$, where $\cdot$ is the standard dot product on $H \sub \RR^{r_1 + r_2}$.
Then the set of Paley-Wiener functions, denoted by $\PW(\Lambda^* \times \CC)$, are functions in $\OO(\Lambda^* \times \CC)$ such that there exists $B \in \RR_+$ with the property that for any $N \in \ZZ^+$, we have
\[
\abs{f(\lambda^*, s)} = O\pr{(1 + \abs{s})^{-N} (1 + \abs{\Lambda^\vee})^{-N} \exp(B \abs{\Re s})}.
\]


\section{Shuffle algebra associated to $\Phi_K$}
\label{sec: shuffle algebra}

\blank Let $\Perm_n$ denote the symmetric group of permutations of $\db{1,n}$. For this paper, we have the following conventions:

\begin{itemize}
    \item For any $n$-tuple $s=(s_1,\dots,s_n)$, we denote $w(s)=(s_{w^{-1}(1)},\dots,s_{w^{-1}(n)})$, so that $(w w')(s) = w(w'(s))$.

    \item We embed $\Perm_n$ in $\GL_n$ via $w\mapsto \begin{pmatrix}e_{w(1)}&\dots&e_{w(n)}\end{pmatrix}$, where $e_i$ is the column vector whose $i$-th entry is 1 and its other entries are 0. Note that this embedding is an injective group homomorphism.
\end{itemize}For $w\in\Perm_n$, let \[\Phi_{K,w}(\lambda^*,s)=\prod_{\substack{1\leqslant i<j\leqslant n\\w(i)>w(j)}}\Phi(\lambda^*_j-\lambda^*_i,s_j-s_i)\]We have the following proposition, which will be helpful later.

\begin{prop}
\label{phi and inverse}
We have that $\Phi_{K,w}(\lambda^*,s)=\Phi_{K,w^{-1}}(-w(\lambda^*),-w(s))$.
\end{prop}

\begin{proof}
Note that 

\begin{align*}
\Phi_{K,w^{-1}}(-w(\lambda^*),-w(s))&=\prod_{\substack{1\leqslant i<j\leqslant n\\w^{-1}(i)>w^{-1}(j)}}\Phi_K((-w(\lambda^*))_i-(-w(\lambda^*))_j,(-w(s))_i-(-w(s))_j)\\
&=\prod_{\substack{1\leqslant i<j\leqslant n\\w^{-1}(i)>w^{-1}(j)}}\Phi_K(w(\lambda^*)_j-w(\lambda^*)_i,w(s)_j-w(s)_i)\\
&=\prod_{\substack{1\leqslant i<j\leqslant n\\w^{-1}(i)>w^{-1}(j)}}\Phi_K(\lambda^*_{w^{-1}j}-\lambda^*_{w^{-1}(i)},s_{w^{-1}j}-s_{w^{-1}(i)})\\
&=\prod_{\substack{1\leqslant w(i)<w(j)\leqslant n\\i>j}}\Phi_K(\lambda^*_{j}-\lambda^*_{i},s_{j}-s_{i})\\
&=\prod_{\substack{1\leqslant j<i\leqslant n\\w(j)>w(i)}}\Phi_K(\lambda^*_{j}-\lambda^*_{i},s_{j}-s_{i})\\
&=\prod_{\substack{1\leqslant i<j\leqslant n\\w(i)>w(j)}}\Phi_K(\lambda^*_{i}-\lambda^*_{j},s_{i}-s_{j})\\
&=\Phi_{K,w}(s)\\
\end{align*}as desired.    
\end{proof}

We define the set of $(m,n)$\textit{-shuffles} as the set $\Sh(m,n)$ of permutations $w\in\Perm_{m+n}$ such that $w(i)<w(j)$ whenever $i,j\in\db{1,m}$ or $i,j\in\db{m+1,m+n}$. For example, $\Sh(2,1)=\{\id,(2 \ 3),(1 \ 2\ 3)\}$. It is straightforward to verify that $\#\Sh(m,n)=\binom{m+n}{m}$.

Now, let us define the \textit{shuffle multiplication } $\shufflemult_{m, n} : \Mer((\Lambda^*\times\CC)^m) \otimes_\CC \Mer((\Lambda^*\times\CC)^n) \to \Mer((\CC \times \Lambda)^{m + n})$ by $F \otimes F' \mapsto F \shufflemult F'$, where
\[
(F \shufflemult F')(z_1, \ldots, z_{m + n}) = \sum_{w \in \Sh(m, n)} F(z_{w(1)}, \ldots, z_{w(m)}) F'(z_{w(m + 1)}, \ldots, z_{w(m + n)}) \cdot \Phi_{K, w\inv}(z_1, \ldots, z_{m + n}).
\]

\begin{prop}\label{prop: shuffle multiplication}
The shuffle multiplication makes $\bigoplus_{n\geqslant0}\Mer((\Lambda^*\times\CC)^n)$ into a graded associative algebra with unit 1.
\end{prop}

\begin{proof}
The proof can be found in \cite{shuffle}. 
\end{proof}

Let $\SH(\Phi_K)_\PW$ be the \textit{Paley-Wiener shuffle algebra associated to }$\Phi_K$, which we define as the subalgebra in $\left(\bigoplus_{n\geqslant0}\Mer((\Lambda^*\times\CC)^n),\shufflemult\right)$ generated by the subspace $\PW( \Lambda^*\times\CC)$.


\section{Fourier transform}
\label{sec: Fourier transform}

\blank We define the Fourier transform to be an operator $\F : \PW((\Lambda^* \times \CC)^n) \to  C_c^\oo(\B^n)$ as
\begin{align*}
    \F f(\lambda^*, s) & = \int_{x \in D} f(x) \abs{x}_\B^s x^{2 \pi i \lambda^*} \frac{\diff x}{\abs{x}_\B},
\end{align*}
where $D \sub \RR_+^{(r_1 + r_2)n}$ is a fundamental domain for $\B^n = (\OO_K^\times)^n \setminus \RR_+^{(r_1 + r_2)n}$.

Similarly, we define the inverse Fourier transform to be an operator $\G : \PW((\Lambda^* \times \CC)^n) \to  C_c^\oo(\B^n)$ as
\[\G g(r) = \frac{1}{(2\pi i)^{(r_1 + 2r_2)n}}\sum_{\lambda^*\in (\Lambda^*)^n}\int_{s \in \sigma_0 + i \RR^n} g(\lambda^*, s) \abs{r}_\B^{-s}r^{-2\pi i\lambda^*} \diff s.\]Here, $\sigma_0$ is chosen such that the integral converges.

It can be checked that $\F\G(g)=g$ and $\G\F(f)=f$. We also have the following generalization of a classical result.

\begin{prop}\label{prop: fourier is iso}

The Fourier transform $\F:H_1 = C^\infty_c(\B) \to \PW(\Lambda^*\times \CC)$ is an isomorphism.
\end{prop}

\begin{proof}
The proof can be adapted from Vol. II, Thm. IX.11 in \cite{fourier}.
\end{proof}


\section{The principal series intertwiners}
\label{sec: intertwiner}

\blank We define
\begin{align*}
    \C : \Lambda^* \times \CC & \to C^\oo(\B) \\
    \C(\lambda^*, s)(a) & = \abs{a}_\B^s a^{2 \pi i \lambda^*} = \abs{a}_\B^s \exp(2 \pi i \brk{\lambda^*, \log a})).
\end{align*}

For $w \in \Perm_n$, we define $U_w= U \cap (w\inv U w)$ and
\begin{align*}
    \ob{M}_w : C_c^\oo(U(\A_K) A_n(K) \doublecoset{\GL_n(\A_K)} \widehat{\K}_n) & \to C^\oo(U(\A_K) A_n(K) \doublecoset{\GL_n(\A_K)} \widehat{\K}_n) \\
    \ob{M}_w(\psi)(g) & = \int_{u \in U_w(\A_K) \setminus U(\A_K)} \psi(w u g) \diff u,
\end{align*}
where we are using the Tamagawa measure on $\A_K$.

We will see that $\diag : \B^n \to U(\A_K) A_n(K) \doublecoset{\GL_n(\A_K)} \widehat{\K}_n$ is a bijection, and we use a twist to identify their smooth functions together:
\begin{align*}
    \delta^{1/2}: C^\oo(\B^n) & \iso C^\oo(U(\A_K) A_n(K) \doublecoset{\GL_n(\A_K)} \K_n) \\
    \vp & \mapsto \left[\diag(a) \mapsto \vp(a) \delta(a)^{1/2}\right],
\end{align*}
where $\delta(a) = \prod_{1 \le i < j \le n} \frac{\abs{a_j}_\B}{\abs{a_i}_\B}$.

The fact that $\diag : \B^n \to U(\A_K) A_n(K) \doublecoset{\GL_n(\A_K)} \widehat{\K}_n$ is an isomorphism also gives us the following proposition, which can be proved by translating the Hall product formula into group-theoretical terms.

\begin{prop}\label{prop: * version}
If $\varphi \in H_1^{\otimes n}$, $g\in \B^n$, and $\~\varphi(g) = \vp(\diag\inv(g))\cdot\delta(g)^{-\frac{1}{2}}$, then
\[
(*_{1^n}(\varphi))(g)=\sum_{\gamma\in B_n(\OO_K) \backslash \GL_n(\OO_K)}\~\varphi(\gamma g).
\]
\end{prop}

Finally, we define
\[
M_w: C_c^\oo(\B^n) \to C^\oo(\B^n)
\]
to be the composition
\[
C_c^\oo(\B^n) \go{\delta^{1/2}} C_c^\oo(U(\A_K) A_n(K) \doublecoset{\GL_n(\A_K)} \K_n) \go{\ob{M}_w} C^\oo(U(\A_K) A_n(K) \doublecoset{\GL_n(\A_K)} \K_n) \go{\delta^{-1/2}} C^\oo(\B^n).
\]
Note that we can also apply $M_w$ to functions in $C^\oo(\B^n)$ that decay fast enough.

One of our main goals is to compute
\[
M_w \C(\lambda^*, s) = M_w\pr{\bigotimes_{i = 1}^n \C(\lambda_i^*, s_i)}
\]
for $\lambda^* \in (\Lambda^*)^n, s \in \CC^n$.

\subsection{The diagonal map is bijective}

\blank In this subsection, we will prove the following proposition.

\begin{prop} \label{prop: copies of Bun1}
    The natural embedding $\RR_+^{r_1 + r_2} \hookrightarrow \RR^{r_1} \times \CC^{r_2} \hookrightarrow \A_K$ induces a bijection between sets
    \begin{align*}
        \B^n & \iso U_n(\A_K) A_n(K) \doublecoset{\GL_n(\A_K)} K_n \\
        \ob{\pl{r}{n}} & \mapsto \ob{\pr{(\id)_{p \in \PP_K}, \diag\pl{r}{n}}},
    \end{align*}
    where $r_i \in \prod_{\nu \in \s} \RR_+$.
\end{prop}

To prove this proposition, we need a few lemmas.

\begin{lemma} \label{lem: p parts cancel}
    For prime ideal $(p) = \pp \sub \OO_K$, we have
    \[
    \GL_n(K_\pp) = B_n(\OO_K[1/p]) \GL_n(\OO_{K, \pp}),
    \]
    where $B_n = U_n A_n$ is the set of lower-triangular matrices.
\end{lemma}

\begin{proof}
    Pick $g \in \GL_n(K_\pp)$.
    We do the following procedure.

    \begin{enumerate}
        \item After right multiply by a matrix in $\GL_n(\OO_{K, \pp})$, we can assume $v_\pp(g_{11}) \le v_\pp(g_{1j})$ for all $j \in \db{1, n}$.

        \item After right multiply by a matrix in $\GL_n(\OO_{K, \pp})$, we can assume $g_{1j} = 0$ for $j > 1$.

        \item Repeat Step 1 and 2, we can assume $g \in U_n(K_\pp)$.

        \item After right multiply by a matrix in $\GL_n(\OO_{K, \pp})$, we can assume $g_{ii} = p^{a_i}$ for some $a_i \in \ZZ$.

        \item After left multiply by a diagonal matrix in $B_n(\OO_K[1/p])$, we can assume $g_{ii} = 1$ for $i \in \db{1, n}$.

        \item After left multiply by a matrix in $U_n(\OO_K[1/p])$, we can assume $v_\pp(g_{i1}) \ge 0$ for $i > 1$.

        \item Repeat Step 6, we can assume $g_{ij} \in \OO_{K, \pp}$ for $i > j$.
    Now the resulting matrix is in $\GL_n(\OO_{K, \pp})$ and we finishes our proof.
    \end{enumerate}

\end{proof}

\begin{lemma} \label{lem: real part}
    We have a bijection between sets
    \begin{align*}
        \RR_+^n & \iso U_n(\RR) \doublecoset{\GL_n(\RR)} O_n \\
        \pl{r}{n} & \mapsto \ob{\diag\pl{r}{n}}.
    \end{align*}
\end{lemma}

\begin{proof}
    To see the injectivity, let us say that $\pl{r}{n}$ and $\pl{r'}{n}$ have the same image, i.e., there exist $u \in U_n(\RR), o \in O_n$ such that $\diag\pl{r'}{n} = u \diag\pl{r}{n} o$.
    Then take their image under the map $A \mapsto A A^T$ gives $\diag\pr{(r_1')^2, \ldots, (r_n')^2} = u \diag\pl{r^2}{n} u^T$, or $u\inv \diag\pr{(r_1')^2, \ldots, (r_n')^2} = \diag\pl{r^2}{n} u^T$.
    However, LHS is lower-triangular, while RHS is upper-triangular, so we see both are diagonal, i.e., $u = \id$, which implies $\diag\pl{r}{n} = \diag\pl{r'}{n}$.

    Now we show the surjectivity.
    We identify $\GL_n(\RR) / O_n$ as the space of positive definite symmetric forms, say $X$, via the map $A \mapsto A A^T$.
    Then the action of $U_n(\RR)$ on $X$ is $u \cdot q = u q u^T$.
    If the matrix of $q$ corresponds to the basis $\bl{v}{n}$, then the matrix $u q u^T$ corresponds to applying $q$ to the basis $u (v_1\ v_2\ \cdots \ v_n)^T$, and we see we can always use such actions to make the basis to be orthogonal, which shows the original map is surjective.
\end{proof}

\begin{lemma} \label{lem: complex part}
    Let $\mathscr{U}_n$ be the set of unitary matrices in $\GL_n(\CC)$, then we have a bijection between sets
    \begin{align*}
        \RR_+^n & \iso U_n(\CC) \doublecoset{\GL_n(\CC)} \mathscr{U}_n \\
        \pl{r}{n} & \mapsto \ob{\diag\pl{r}{n}}.
    \end{align*}
\end{lemma}

\begin{proof}
    Just mimic Lemma~\ref{lem: real part}.
\end{proof}

Now we are at a position to prove Proposition~\ref{prop: copies of Bun1}.

\begin{proof} [Proof of Proposition~\ref{prop: copies of Bun1}]
    The map is clearly well-defined.
    We first show that it is surjective.

    Pick any $\ob{g} = \ob{\pr{(g_\pp)_{\pp \in \Spec(\OO_K)}, (g_\nu)_{\nu \in \s}}} \in U_n(\A_K) A_n(K) \doublecoset{\GL_n(\A_K)} K_n$, where $g_\pp \in \GL_n(\OO_{K, \pp})$ for $\pp \in \Spec(\OO_K)-\{\p_1,\dots,\p_m\}$ for some primes $\p_1,\dots,\p_m$.
    For $\pp \in \{\p_1,\dots,\p_m\}$, from Lemma~\ref{lem: p parts cancel} we can write $g_\pp = u_\pp a_\pp h_\pp \in U_n(\OO_K[1/p]) A_n(\OO_K[1/p]) \GL_n(\OO_{K, \pp})$ (where $(p) = \pp$).
    Also, from Lemma~\ref{lem: real part} and Lemma~\ref{lem: complex part}, for each $\nu \in \s$, we can write $g_\nu = u_\nu a_\nu o_\nu \in U_n(K_\nu) A_n(\RR_+) O_\nu$.
    Now take $a = \prod_{\pp \in \{\p_1,\dots,\p_m\}} a_\pp \in A_n(K)$, then
    \begin{align*}
        \ob{g} & = \ob{a\inv g}\\
            & = \ob{\pr{\pr{(a\inv u_\pp a) \pr{\prod_{\pp' \in \PP' - \br{\pp}} a_{\pp'}\inv} h_\pp}_{\pp \in \{\p_1,\dots,\p_m\}}, \pr{a\inv g_\pp}_{\pp \in \Spec(\OO_K) - \{\p_1,\dots,\p_m\}}, \pr{(a\inv u_\nu a) a\inv a_\nu o_\nu}_{\nu \in \s}}} \\
            & = \ob{\pr{(\id)_{\pp \in \Spec(\OO_K)}, (a\inv a_\nu)_{\nu \in \s}}},
    \end{align*}
    which is in the image of the map.

    Now we show that the map is injective.
    To do this, define map $\vp : K^\times \to \prod_{\nu \in \s} \RR_+, x \mapsto \abs{\nu(x)}$ and assume $\ob{\pl{r}{n}}$ has the same image as $\ob{\pl{r'}{n}}$, where $r_i, r_i' \in \prod_{\nu \in \s} \RR_+$.
    Then there exists
    \[
    \pr{(u_\pp)_{\pp \in \Spec(\OO_K)}, (u_\nu)_{\nu \in \s}} \in U_n(\A_K), \diag\pl{x}{n} \in A_n(K), \pr{(h_\pp)_{\pp \in \Spec(\OO_K)}, (h_\nu)_{\nu \in \s}} \in K_n
    \]
    such that
    \[
    \pr{(\id)_{\pp \in \Spec(\OO_K)}, \diag\pl{r'}{n}} = \pr{(u_\pp)_{\pp \in \Spec(\OO_K)}, (u_\nu)_{\nu \in \s}} \diag\pl{x}{n} \pr{(\id)_{\pp \in \Spec(\OO_K)}, \diag\pl{r}{n}} \pr{(h_\pp)_{\pp \in \Spec(\OO_K)}, (h_\nu)_{\nu \in \s}}.
    \]

    Now for $\pp \in \Spec(\OO_K)$, we have $u_\pp \diag\pl{x}{n} h_\pp = \id$, so both $u_\pp \diag\pl{x}{n}$ and $\diag\pl{x\inv}{n} u_\pp\inv$ lie in $\GL_n(\OO_{K, \pp})$, which implies $x_i \in \OO_{K, \pp}^\times$.
    This shows $x_i \in K^\times \cap \bigcap_{\pp \in \Spec(\OO_K)} \OO_{K, \pp}^\times = \OO_K^\times$.

    Now notice that in $\prod_{\nu \in \s} U_n(K_\nu) \doublecoset{\GL_n(K_\nu)} O_\nu$, we have
    \begin{align*}
        \ob{\diag\pl{r'}{n}} & = \ob{\diag\pl{x}{n} \diag\pl{r}{n}} \\
            & = \ob{\diag(\vp(x_i) r_i)_{i = 1}^n},
    \end{align*}
    so from Lemma~\ref{lem: real part} and Lemma~\ref{lem: complex part} we see $r_i' = \vp(x_i) r_i$, which implies $\ob{r_i'} = \ob{r_i}$ in $\B$ for all $i \in \db{1, n}$, hence shows that the original map is injective and finishes our proof.
\end{proof}

\subsection{Intertwiner of the exponential function}

\blank Let us compute the result of applying the intertwiner to the function $\C(\lambda^*,s)(a)$. In particular, we claim the following.

\begin{prop}
$M_w(\C(\lambda^*,s))(a)=\C(w^{-1}(\lambda^*),w^{-1}(s))(a)\cdot\Phi_{K,w^{-1}}(\lambda^*,s)$.
\end{prop}

 Let us show this. We will first deal with the case $n = 2, w = (1 2)$.
To simplify the notations, we will compute
\[
M_w\pr{\C(\lambda^*, s) \ten \C(\mu^*, t)}(a, b).
\]

To do this, we will first recall the construction of the inverse map of $\diag : \B^n \iso U_n(\A_K) A_n(K) \doublecoset{\GL_n(\A_K)} \widehat{\K}_n$.
Let $\s$ denote the set of Archimedean places.
Take any $\ob{g} = \ob{\pr{(g_\pp)_{\pp \in \Spec(\OO_K)}, (g_\nu)_{\nu \in \s}}} \in U_n(\A_K) A_n(K) \doublecoset{\GL_n(\A_K)} \widehat{\K}_n$.
We write $g_\pp = u_\pp a_\pp k_\pp \in U_n(K_\pp) A_n(\br{p^\ZZ}) \GL_n(\OO_{K, \pp}), g_\nu = u_\nu a_\nu k_\nu \in U_n(K_\nu) A_n(\RR_+) O_\nu$, where $O_\nu$ is the maximal compact subgroup of $K_\nu$, and $(p) = \pp$.
Then $g$ corresponds to
\[
\pr{\prod_{\pp \in \Spec(\OO_K)} a_\pp\inv} \ob{\pr{a_\nu}}_{\nu \in \s} = \ob{\pr{a_\nu \prod_{\pp \in \Spec(\OO_K)} \abs{\nu(a_\pp)}\inv}}_{\nu \in \s}
\]
in $\B^n$ (after dediagonalizing).
From this and the fact that $\C(\lambda^*, s)$ is a group homomorphism we see we can compute the integration on each place separately, then multiply them together to get the integration result.
We split the computation into 3 parts.

\Skip \textbf{Part 1}: $\nu \in \s_\RR$ (the set of real archimedean places).
We want to find a explicit inverse map of $\diag : \RR_+^2 \to U_2(\RR) \doublecoset{\GL_2(\RR)} O_2$.
To do this, pick $\ob{\smat{a}{b}{c}{d}} \in U_2(\RR) \doublecoset{\GL_2(\RR)} O_2$, then we first do map $A \mapsto A A^T$ to map it to a symmetric form $\smat{a^2 + b^2}{ac + bd}{ac + bd}{c^2 + d^2}$, and then note that in symmetric form $\smat{x}{y}{y}{z}$, after letting $v_2' = v_2 - (y / x) v_1$, the symmetric form becomes $\smat{x}{0}{0}{z - y^2 / x}$.
Thus, under this operation, the inverse map sends $\ob{\smat{a}{b}{c}{d}} \in U_2(\RR) \doublecoset{\GL_2(\RR)} O_2$ to
\[
\pr{\rt{a^2 + b^2}, \rt{c^2 + d^2 - \frac{(ac + bd)^2}{a^2 + b^2}}} = \pr{\rt{a^2 + b^2}, \frac{\abs{\det\smat{a}{b}{c}{d}}}{\rt{a^2 + b^2}}} \in \RR_+^2.
\]

Now let us start integrating.
We have $U_w(\RR) = \br{\id}$ and $U_w(\RR) \setminus U(\RR) = \br{\smat{1}{0}{\ast}{1}}$, and apply the integral to $(a, b)$ the place $\nu$ gives contribution
\begin{align*}
        & a_\nu^{1/2} b_\nu^{-1/2} \int_{x \in \RR} \pr{\delta^{1/2} \cdot \C(\lambda^*, s) \ten \C(\mu^*, t)}\pr{\mat{0}{1}{1}{0} \mat{1}{0}{x}{1} \mat{a_\nu}{0}{0}{b_\nu}} \diff x \\
    =\ & a_\nu^{1/2} b_\nu^{-1/2} \int_{x \in \RR} \pr{\delta^{1/2} \cdot \C(\lambda^*, s) \ten \C(\mu^*, t)} \mat{a_\nu x}{b_\nu}{a_\nu}{0} \diff x \\
    =\ & a_\nu^{1/2} b_\nu^{-1/2} \int_{x \in \RR} \pr{\rt{a_\nu^2 x^2 + b_\nu^2}}^{-1/2} \pr{\frac{a_\nu b_\nu}{\rt{a_\nu^2 x^2 + b_\nu^2}}}^{1/2} \pr{\C(\lambda^*, s) \ten \C(\mu^*, t)}\pr{\rt{a_\nu^2 x^2 + b_\nu^2}, \frac{a_\nu b_\nu}{\rt{a_\nu^2 x^2 + b_\nu^2}}} \diff x \\
    =\ & a_\nu^{1/2} b_\nu^{-1/2} \int_{x = -\oo}^\oo \pr{\C(\lambda^*, s - 1/2) \ten \C(\mu^*, t + 1/2)}\pr{\rt{a_\nu^2 x^2 + b_\nu^2}, \frac{a_\nu b_\nu}{\rt{a_\nu^2 x^2 + b_\nu^2}}} \diff x \\
    =\ & a_\nu^{1/2} b_\nu^{-1/2} a_\nu^{t + 1/2 + 2\pi i \mu^*_\nu} b_\nu^{t + 1/2 + 2\pi i \mu^*_\nu} \int_{x = -\oo}^\oo \pr{a_\nu^2 x^2 + b_\nu^2}^{\frac{s - t - 1 + 2\pi i(\lambda^*_\nu - \mu^*_\nu)}{2}} \diff x \\
    =\ & a_\nu^{t + 1 + 2 \pi i \mu^*_\nu} b_\nu^{t + 2 \pi i \mu^*_\nu} b_\nu^{s - t - 1 + 2\pi i(\lambda^*_\nu - \mu^*_\nu)} \frac{b_\nu}{a_\nu} \rt{\pi} \frac{\Gamma\pr{\frac{t - s}{2} + \pi i (\mu^*_\nu - \lambda^*_\nu)}}{\Gamma\pr{\frac{t - s + 1}{2} + \pi i (\mu^*_\nu - \lambda^*_\nu)}} \\
    =\ & a_\nu^{t + 2 \pi i \mu^*_\nu} b_\nu^{s + 2 \pi i \lambda^*_\nu} \rt{\pi} \frac{\Gamma\pr{\frac{t - s}{2} + \pi i (\mu^*_\nu - \lambda^*_\nu)}}{\Gamma\pr{\frac{t - s + 1}{2} + \pi i (\mu^*_\nu - \lambda^*_\nu)}}
\end{align*}
when $\Re(t - s) > 0$ since
\[
\int_{x \in \RR} \pr{a x^2 + b}^s \diff x = b^s \rt{\frac{b}{a}} \rt{\pi} \frac{\Gamma(-s - 1/2)}{\Gamma(-s)}
\]
when $a, b \in \RR_+, \Re s < -1/2$.

\Skip \textbf{Part 2}: $\nu \in \s_\CC$ (the set of complex archimedean places).
Let us first get the inverse map of $\diag : \RR_+^2 \iso U_2(\CC) \doublecoset{\GL_2(\CC)} \mathscr{U}_2$, where $\mathscr{U}_2$ is the unitary group.
Take any $\ob{\smat{a}{b}{c}{d}} \in U_2(\CC) \doublecoset{\GL_2(\CC)} \mathscr{U}_2$.
Then the map $A \mapsto A A^*$ sends it to $\smat{\abs{a}^2 + \abs{b}^2}{a \ob{c} + b \ob{d}}{\ob{a} c + \ob{b} d}{\abs{c}^2 + \abs{d}^2}$, and note that in Hermitian form $\smat{x}{y}{\ob{y}}{z}$, after letting $v_2' = v_2 - (\ob{y}/x) v_1$, we see the matrix becomes $\smat{x}{0}{0}{z - \abs{y}^2 / x}$.
Thus, under this operation, the inverse map sends $\ob{\smat{a}{b}{c}{d}} \in U_2(\CC) \doublecoset{\GL_2(\CC)} \mathscr{U}_2$ to
\[
\pr{\rt{\abs{a}^2 + \abs{b}^2}, \rt{\abs{c}^2 + \abs{d}^2 - \frac{\abs{a\ob{c} + b\ob{d}}^2}{\abs{a}^2 + \abs{b}^2}}} = \pr{\rt{\abs{a}^2 + \abs{b}^2}, \frac{\abs{\det\smat{a}{b}{c}{d}}}{\rt{\abs{a}^2 + \abs{b}^2}}} \in \RR_+^2.
\]

Then similarly, apply the place $\nu$ to $(a, b)$ yields contribution
\begin{align*}
        & a_\nu b_\nu\inv \int_{x \in \CC} \pr{\delta^{1/2} \cdot \C(\lambda^*, s) \ten \C(\mu^*, t)}\pr{\mat{0}{1}{1}{0} \mat{1}{0}{x}{1} \mat{a_\nu}{0}{0}{b_\nu}} \diff x \\
    =\ & a_\nu b_\nu\inv \int_{x \in \CC} \pr{\delta^{1/2} \cdot \C(\lambda^*, s) \ten \C(\mu^*, t)} \mat{a_\nu x}{b_\nu}{a_\nu}{0} \diff x \\
    =\ & a_\nu b_\nu\inv \int_{x \in \CC} \pr{\C(\lambda^*, s - 1/2) \ten \C(\mu^*, t + 1/2)}\pr{\rt{a_\nu^2 \abs{x}^2 + b_\nu^2}, \frac{a_\nu b_\nu}{\rt{a_\nu^2 \abs{x}^2 + b_\nu^2}}} \diff x \\
    =\ & a_\nu b_\nu\inv a_\nu^{2t + 1 + 2\pi i \mu^*_\nu} b_\nu^{2t + 1 + 2\pi i \mu^*_\nu} \int_{x \in \CC} \pr{a_\nu^2 \abs{x}^2 + b_\nu^2}^{s - t - 1 + \pi i(\lambda^*_\nu - \mu^*_\nu)} \diff x \\
    =\ & 2 \pi a_\nu^{2 t + 2 + 2 \pi i \mu^*_\nu} b_\nu^{2t + 2\pi i \mu^*_\nu} \int_{r = 0}^\oo r \pr{a_\nu^2 r^2 + b_\nu^2}^{s - t - 1 + \pi i(\lambda^*_\nu - \mu^*_\nu)} \diff r \\
    =\ & 2 \pi a_\nu^{2 t + 2 + 2 \pi i \mu^*_\nu} b_\nu^{2t + 2\pi i \mu^*_\nu} \frac{b_\nu^{2 s - 2 t + 2 \pi i(\lambda^*_\nu - \mu^*_\nu)}}{2 a_\nu^2 (t - s + \pi i(\mu^*_\nu - \lambda^*_\nu))} \\
    =\ & a_\nu^{2 t + 2 \pi i \mu^*_\nu} b_\nu^{2s + 2\pi i \lambda^*_\nu} \pi \frac{1}{t - s + \pi i(\mu^*_\nu - \lambda^*_\nu)}
\end{align*}
when $\Re(t - s) > 0$.
Also, note that we should multiply by a factor of 2 for each complex place that comes from the Tamagawa measure.

\Skip \textbf{Part 3}: finite place $\pp \in \Spec(\OO_K)$, where $\pp = (p) \sub \OO_K$ is a prime ideal.
First let us find out the $A$ in $\smat{c p^i}{1}{1}{0} = u A k \in U_2(K_\pp) A_n(\br{p^\ZZ}) \GL_2(\OO_{K, \pp})$ for $i \in \ZZ, c \in \OO_{K, \pp}^\times$.
For $i \ge 0$ we can take $A = \id$.
For $i < 0$ we do column operation to make the matrix becomes $\smat{c p^i}{0}{1}{-c\inv p^{-i}}$, then we see $A = \smat{p^i}{0}{0}{p^{-i}}$.

Now let us start integrating.
Denote the natural embedding $K \to \prod_{\nu \in \s} K_\nu, x \mapsto (\nu(x))_{\nu \in \s}$ by $\iota$.
Apply the place $\pp$ to $(a, b)$ yields contribution
\begin{align*}
        & \int_{x \in K_\pp} \pr{\delta^{1/2} \cdot \C(\lambda^*, s) \ten \C(\mu^*, t)}\pr{\mat{0}{1}{1}{0} \mat{1}{0}{x}{1}} \diff x \\
    =\ & \int_{x \in K_\pp} \pr{\C(\lambda^*, s - 1/2) \ten \C(\mu^*, t + 1/2)}\mat{x}{1}{1}{0} \diff x \\
    =\ & \sum_{j \in \ZZ} \int_{x \in p^j \OO_{K, \pp}^\times} \pr{\C(\lambda^*, s - 1/2) \ten \C(\mu^*, t + 1/2)}\mat{x}{1}{1}{0} \diff x \\
    =\ & 1 + \sum_{j < 0} \N(p)^{-j} \pr{1 - \frac{1}{\N(p)}} \pr{\C(\lambda^*, s - 1/2) \ten \C(\mu^*, t + 1/2)}\mat{\abs{\nu(p)}^{-j}}{0}{0}{\abs{\nu(p)}^j}_{\nu \in \s} \\
    =\ & 1 + \sum_{j < 0} \N(p)^{-j} \pr{1 - \frac{1}{\N(p)}} \exp(j\cdot 2 \pi i (\mu^* - \lambda^*) \cdot \log \abs{\iota(p)}) \N(p)^{j (t - s + 1)} \\
    =\ & 1 + \pr{1 - \frac{1}{\N(p)}} \sum_{j > 0} \frac{1}{\abs{\iota(p)}^{j \cdot2\pi i (\mu^* - \lambda^*)} \N(p)^{j (t - s)}} \\
    =\ & 1 + \pr{\frac{\N(p) - 1}{\N(p)}} \frac{1}{\abs{\iota(p)}^{2 \pi i(\mu^* - \lambda^*)} \N(p)^{t - s} - 1} \\
    =\ & \frac{\abs{\iota(p)}^{2 \pi i(\mu^* - \lambda^*)} \N(p)^{t - s + 1} - 1}{\abs{\iota(p)}^{2 \pi i(\mu^* - \lambda^*)} \N(p)^{t - s + 1} - \N(p)}
\end{align*}
when $\Re(t - s) > 0$.
Also, note that we should multiply by a factor of $\disc(\OO_K)^{-1/2}$ at the end that comes from the Tamagawa measure.

Combine them together, after letting $s = s_2 - s_1, \lambda^* = \lambda^*_2 - \lambda^*_1$, we get
\begin{align*}
        & M_w(\C(\lambda^*_1, s_1) \ten \C(\lambda^*_2, s_2))(a_1, a_2) \\
    =\ & (\C(\lambda^*_2, s_2) \ten \C(\lambda^*_1, s_1))(a_1, a_2) \cdot \disc(\OO_K)^{-1/2} \cdot \pi^{\frac{r_1}{2}} \cdot (2\pi)^{r_2} \\
        & \cdot \prod_{\nu \in \s_\RR} \frac{\Gamma\pr{\frac{s}{2} + \pi i \lambda^*_\nu}}{\Gamma\pr{\frac{s + 1}{2} + \pi i \lambda^*_\nu}} \cdot \prod_{\nu \in \s_\CC} \frac{1}{s + \pi i \lambda^*_\nu} \\
        & \cdot \prod_{(p) \in \Spec(\OO_K)} \frac{\abs{\iota(p)}^{2 \pi i \lambda^*} \N(p)^{s + 1} - 1}{\abs{\iota(p)}^{2 \pi i \lambda^*} \N(p)^{s + 1} - \N(p)}
\end{align*}
when $\Re s > 1$ (for each component we need $\Re s > 0$, but for the product to converge, we need $\Re s > 1$).

As a result, for $\Re(s_2 - s_1) > 1$, we see
\[
M_w(\C(\lambda^*_1, s_1) \ten \C(\lambda^*_2, s_2))(a_1, a_2) = (\C(\lambda^*_2, s_2) \ten \C(\lambda^*_1, s_1))(a_1, a_2) \Phi_K(\lambda_2^* - \lambda_1^*, s_2 - s_1).
\]

To figure out $M_w \C(\lambda^*, s)$ in the general case, we need some extra works. 
For $w \in \Perm_n$ we use $\I_w$ to denote the set $\br{(i, j) \in \db{1, n}^2 : i < j, w(i) > w(j)}$.
First, let us prove a lemma.

\begin{lemma} \label{lem: describe Uw}
    $(a_{ij})_{i,j}\in U_w$ if and only if $a_{ij}=0$ and $a_{w^{-1}(i)w^{-1}(j)}=0$ when $i<j$; or equivalently, $(a_{ij}) \in U$ and $a_{ij} = 0$ for $(j, i) \in \I_w$.
\end{lemma}

\begin{proof}
    Note that

    \begin{align*}
    (a_{ij})_{i,j}\in U_w&\Longleftrightarrow (a_{ij})_{i, j} \in U \text{ and } w(a_{ij})_{i,j}w^{-1}\in U\\
    &\Longleftrightarrow \br{(a_{ij})_{i, j}, (a_{w^{-1}(i)w^{-1}(j)})_{i,j}} \sub U\\
    &\Longleftrightarrow a_{ij} = a_{w^{-1}(i)w^{-1}(j)} = 0\text{ for }i<j\\
    \end{align*}as desired.
\end{proof}

\begin{cor} \label{cor: Uw for transposition}
    If $t = (i\ j)$ is a transposition with $i < j$, then $U_t = \{(a_{ij})_{i,j}\in U : a_{ji} = 0\}$.
\end{cor}

\begin{cor} \label{cor: representative for transposition}
    If $t = (i\ j)$ is a transposition with $i<j$, then each equivalence class in $U_t\backslash U$ has a representative $(a_{i'j'})_{i',j'}$ with $a_{i'i'}=1$ and $a_{i'j'}=0$ for all $i'\neq j'$ except for possibly $(i',j')=(j,i)$.
\end{cor}

\noindent\textit{Proof.} Let $b = (b_{k\ell})\in U_t\backslash U$ and let $b^{-1} =(c_{k\ell})_{k,\ell}$ be its inverse. Let $d =(d_{k\ell})_{k,\ell}\in U_t$ be such that $d_{k\ell} = c_{k\ell}$ for $(k,\ell)\neq (i,j)$ and $d_{ij}=0$. Note that $h = db$ is also a representative of $[b]$, but now $h_{kk} = 1$ and $h_{k\ell} = 0$ for all $k\neq \ell$ except for possibly $\ell = i$. Now, consider $m = (m_{ij})\in U_t$, with $m_{kk} = 1$, $m_{ki} = -h_{ki}$ for $k\neq j$, and all the entries of $m$ equal to 0. Note that, if $n = mh$, we have that $n_{kk} = 1$, $n_{k\ell} = \sum_rm_{kr}h_{r\ell} = 0$ if $k\neq\ell$ and $\ell\neq i$, and $n_{ki} = \sum_rm_{kr}h_{ri}=h_{ki}-h_{ki}=0$ if $k\neq j$. That is, $n = mh = mdb$ is the desired representative of $[b]$ with $n_{kk} = 1$, $n_{k\ell} = 0$ if $k\neq\ell$ except for possibly $(k,\ell)\neq (j,i)$, as desired. 

\begin{prop} \label{prop: M contravariant wrt good transpositions}
    If $w \in \Perm_n$ and $t$ is a transposition of the form $(\ell\ \ell + 1)$ such that $\#\I_{wt} > \#\I_w$, then we have $M_t \circ M_w = M_{wt}$.
\end{prop}

\begin{proof}
    For $w \in \Perm_n$, we let $U^w = w\inv U w$.
    It suffices to show that $\ob{M}_t \circ \ob{M}_w = \ob{M}_{wt}$.
    We have
    \begin{align*}
        \ob{M}_t \ob{M}_w \psi(g) & = \int_{v \in U_t \setminus U} \int_{u \in U_w \setminus U} \psi(w u t v g) \diff u \diff v \\
            & = \int_{v \in U_t \setminus U, u \in U_w \setminus U} \psi(w t (t u t) v g) \diff u \diff v \\
            & = \int_{u \in (U^{wt} \cap U^t) \setminus U^t, v \in (U^t \cap U) \setminus U} \psi(w t u v g) \diff u \diff v
    \end{align*}

    Now we introduce a lemma.

    \begin{lemma} \label{lem: bijection between left cosets inside Ut}
        For transposition $t = (\ell\ \ell + 1)$ and $w \in \Perm_n$ such that $\#\I_{wt} > \#\I_w$, the natural map between sets
        \[
        (U^{wt} \cap U) \setminus (U^t \cap U) \to (U^{wt} \cap U^t) \setminus U^t
        \]
        is a bijection (note that $U^{wt} \cap U \subset U^t$).
    \end{lemma}

    \begin{proof}
        We first show that it is injective.
        Take any $g, g' \in U^t \cap U$ such that there exists $h \in U^{wt} \cap U^t$ with $h g = g'$.
        Then $h = g' g\inv \in U^t \cap U$ (note that both $U$ and $U^t$ are groups).
        Thus, $\ob{g} = \ob{g'}$ in $(U^{wt} \cap U) \setminus (U^t \cap U)$ and so the map is injective.

        Now we prove the surjectivity.
        For $x \in \A_K$, let $A(x)$ be the matrix with $A(x)_{ii} = 1$ for $i \in \db{1, n}$ and $A(x)_{\ell, \ell + 1} = x, A(x)_{ij} = 0$ for all other $ij$'s.
        Then we claim that $A(x) \in U^{wt}$ for any $x \in \A_K$.
        This is because the condition $\#\I_{wt} > \#\I_w$ shows $w(\ell) < w(\ell + 1)$, and $U^{wt}$ consists of matrices $A$ with $A_{(wt)\inv(i), (wt)\inv(j)} = 0$ for $i < j$.
        Thus, $A(x) \not\in U^{wt}$ only when there exists $i < j$ such that $(wt)\inv(i) = \ell, (wt)\inv(j) = \ell + 1$, i.e., $i = wt(\ell) = w(\ell + 1), j = wt(\ell + 1) = w(\ell)$.
        However, then the assumption $w(\ell) < w(\ell + 1)$ gives a contradiction.

        Now given any $\ob{g} \in (U^{wt} \cap U^t) \setminus U^t$, we have $\ob{g} = \ob{Y(-g_{\ell, \ell + 1}) g}$, and the latter is in the image of the map, so we are done.
    \end{proof}

    Thus, we see
    \[
    \ob{M}_t \ob{M}_w \psi(g) = \int_{u \in (U^{wt} \cap U) \setminus (U^t \cap U), v \in (U^t \cap U) \setminus U} \psi(w t u v g) \diff u \diff v.
    \]

    Therefore, it suffices to show that
    \[
    \int_{u \in (U^{wt} \cap U) \setminus (U^t \cap U), v \in (U^t \cap U) \setminus U} \psi(w t u v g) \diff u \diff v = \int_{x \in (U^{wt} \cap U) \setminus U} \psi(wt x g) \diff x.
    \]

    Note that from Corollary~\ref{cor: representative for transposition} we see there is a natural bijection between sets
    \[
    Y \to (U^t \cap U) \setminus U
    \]
    where $Y \sub U$ is the subset $\br{g \in U : g_{ij} = 0 \text{ for } i \ne j \text{ except } (i, j) = (\ell + 1, \ell)}$.

    \begin{lemma} \label{lem: desired map}
        There is a bijection between sets
        \[
        (U^{wt} \cap U) \setminus (U^t \cap U) \times (U^t \cap U) \setminus U \to (U^{wt} \cap U) \setminus (U^t \cap U) \times Y \go{\cdot} (U^{wt} \cap U) \setminus U,
        \]
        where $\cdot$ denote the usual product.
    \end{lemma}

    \begin{proof}
        It suffices to show that the $\cdot$ part is a bijection.
        We first show the injectivity.
        Let $Y(x) = A(x)^T$, where $A(x)$ is the matrix defined in Lemma~\ref{lem: bijection between left cosets inside Ut}.
        Take any $g, g' \in U^t \cap U, x, x' \in \A_K, h \in U^{wt} \cap U$ such that $h g Y(x) = g' Y(x')$.
        Then by looking at the $(\ell + 1, \ell)$ entry we see $x = x'$, so $h g = g'$ and $\ob{g} = \ob{g'}$ in $(U^{wt} \cap U) \setminus (U^t \cap U)$, as desired.

        Now we show the surjectivity.
        Take any $\ob{g} \in (U^{wt} \cap U) \setminus U$, then we see it is the image of $(\ob{g Y(-g_{\ell + 1, \ell})}, Y(g_{\ell + 1, \ell}))$, so we are done.
    \end{proof}

    We now let $G = (U^{wt} \cap U) \setminus U$ and $H = (U^{wt} \cap U) \setminus (U^t \cap U)$, then Lemma~\ref{lem: desired map} shows $Y$ is a fundamental domain for $H \setminus G$.
    From knowledge about Haar measure (for example, see \cite{mathoverflow}) we see for any $f \in C_c(G)$, we have
    \[
    \int_Y \int_H f(h y) \diff \mu_H \diff \mu_{H \setminus G} = c \int_G f(g) \diff \mu_G
    \]
    for some constant $c \in \RR$. We now take $f$ to be any measurable function $G \to \br{0, 1}$, then after taking any fundamental region for $H$ we see $c = 1$ by Cavalieri's principle, which finishes the proof.
\end{proof}

\begin{cor} \label{cor: M contravariant when everything is good}
    For $w, w' \in \Perm_n$ with $\#\I_{w w'} = \#\I_w + \#\I_{w'}$, we have $M_{w'} M_w = M_{w w'}$.
\end{cor}

\begin{proof}
    We use induction on $\#\I_{w'}$.
    When $\#\I_{w'} = 1$, the statement follows from Proposition~\ref{prop: M contravariant wrt good transpositions}, so we assume $\#\I_{w'} > 1$.
    Then we write $w' = t w_0$, where $t$ is a transposition of adjacent elements and $\#\I_{w_0} = \#\I_{w'} - 1$.
    Then we have $\#\I_{wt} = \#\I_w + 1$.
    Now from Proposition~\ref{prop: M contravariant wrt good transpositions} and induction hypothesis we see
    \begin{align*}
        M_{w'} M_w & = M_{w_0} M_t M_w \\
            & = M_{w_0} M_{wt} \\
            & = M_{w t w_0} \\
            & = M_{w w'},
    \end{align*}
    as desired.
\end{proof}

\begin{lemma} \label{lem: Phi is covariant is some sense}
    For $w, w' \in \Perm_n, \lambda^* \in (\Lambda^*)^n, z \in (\Lambda^* \times \CC)^n$, we have
    \[
    \Phi_{K, w'}(w(z)) \Phi_{K, w}(z) = \Phi_{K, w' w}(z).
    \]
\end{lemma}

\begin{proof}
    We first show this for $w' = t = (\ell\ \ell + 1)$ being a transposition of adjacent elements.
    To do this, we split into 2 cases.

    \Skip Case 1: $\#\I_{tw} > \#\I_w$.
    Then we have
    \begin{align*}
        \Phi_{K, w'}(w(z))) \Phi_{K, w}(z) & = \Phi_K(z_{w\inv(\ell + 1)} - z_{w\inv(\ell)}) \Phi_{K, w}(z) \\
            & = \Phi_{K, w' w}(z)
    \end{align*}
    because $w\inv(\ell + 1) > w\inv(\ell)$ from assumption that $\#\I_{tw} > \#\I_w$, and it is the only different of $\Phi_{K, w}$ and $\Phi_{K, tw}$ since no extra difference is allowed (we have $\#\I_{tw} = \#\I_w + 1$).

    \Skip Case 2: $\#\I_{tw} < \#\I_w$.
    Then we have $w = tw_0$ for some $w_0$ that satisfies $\#\I_{w_0} < \#\I_w$.
    As a result, we see
    \begin{align*}
        \Phi_{K, w'}(w(z)) \Phi_{K, w}(z) & = \Phi_K(z_{w_0\inv t(\ell + 1)} - z_{w_0\inv t(\ell)}) \Phi_{K, tw_0}(z) \\
            & = \Phi_K(z_{w_0\inv(\ell)} - z_{w_0\inv(\ell + 1)}) \Phi_K(z_{w_0\inv(\ell + 1)} - z_{w_0\inv(\ell)}) \Phi_{K, w_0}(z) \\
            & = \Phi_{K, tw}(z)
    \end{align*}
    because $\Phi_K(z)\inv = \Phi_K(-z)$.

    \Skip Now we deal with the general case.
    We use induction on $\#\I_{w'}$.
    When $\#\I_{w'} = 1$, then $w'$ is a transposition of adjacent elements and we are done from previous discussion, so we assume $\#\I_{w'} > 1$.
    Then we can write $w' = w_0 t$ for some transposition of adjacent elements $t$ and $\#\I_{w_0} = \#\I_{w'} - 1$.
    Then we have
    \begin{align*}
        \Phi_{K, w'}(w(z)) \Phi_{K, w}(z) & = \Phi_{K, w_0 t}(w(z)) \Phi_{K, w}(z) \\
            & = \Phi_{K, w_0}(tw(z)) \Phi_{K, t}(w(z)) \Phi_{K, w}(z) \\
            & = \Phi_{K, w_0}(tw(z)) \Phi_{K, tw}(z) \\
            & = \Phi_{w_0 t w}(z) \\
            & = \Phi_{w' w}(z),
    \end{align*}
    from induction hypothesis and the case when $w'$ is a transposition of adjacent elements, as desired.
\end{proof}

Now for $w \in \Perm_n$, let $\CC_w^n = \br{\pl{s}{n} \in \CC^n : \Re(s_j - s_i) > 1 \text{ for } (i, j) \in \I_w}$.
We are at a position to establish the integration result for the general case.

\begin{prop} \label{prop: key integral for general n}
    Let $w \in \Perm_n, s \in \CC_{w\inv}^n, \lambda^* \in (\Lambda^*)^n$, then
    \[
    M_w \C(\lambda^*, s) = \C(w\inv(\lambda^*), w\inv(s)) \Phi_{K, w\inv}(\lambda^*, s).
    \]
\end{prop}

\begin{proof}
    We use induction on $\#\I_w$.
    If $\#\I_w = 1$, then $w$ is a transposition of adjacent elements, so from previous discussion we see the equation follows.
    Thus, we assume $\#\I_w > 1$.
    Then we can write $w = t w_0$ with $t$ being a transposition of adjacent elements and $\#\I_{w_0} = \#\I_w - 1$.
    Then
    \begin{align*}
        M_w \C(\lambda^*, s) & = M_{w_0} M_t(\C(\lambda^*, s)) \\
            & = M_{w_0} \C(t(\lambda^*), t(s)) \Phi_{K, t}(\lambda^*, s) \\
            & = \C(w_0\inv t(\lambda^*), w_0\inv t(s)) \Phi_{K, w_0\inv}(t(\lambda^*), t(s)) \Phi_{K, t}(\lambda^*, s) \\
            & = \C(w\inv(\lambda^*), w\inv(s)) \Phi_{K, w_0\inv t}(\lambda^*, s) \\
            & = \C(w\inv(\lambda^*), w\inv(s)) \Phi_{K, w\inv}(\lambda^*, s),
    \end{align*}
    from Proposition~\ref{prop: M contravariant wrt good transpositions}, the case when $w$ is a transposition of adjacent elements, the induction hypothesis, and Lemma~\ref{lem: Phi is covariant is some sense}, as desired.
\end{proof}

\subsection{Intertwiner and the Fourier transform}

\blank Let us compute the result of composing the intertwiner and the Fourier transform. More explicitly, we claim the following.

\begin{prop}
\label{prop: 5.21}
    $\F(M_w(\varphi))(\lambda^*,s)=\F(\varphi)(w(\lambda^*),w(s))\cdot \Phi_{K,w}(\lambda^*,s)$ 
\end{prop}

\begin{proof}
Choose $F$ such that $\varphi=\mathcal{G}F$. By \ref{phi and inverse}, we have that

\begin{align*}
\F(M_w(\varphi))(\lambda^*,s)
&=\int_{a\in D}\int_{u\in U_w(\AA_K)\backslash U(\AA_K)}(\mathcal{G} F)(wua)\diff u\C(\lambda^*,s)(a)\diff^* a\\
&=\int_{a\in D}\int_{u\in U_w(\AA_K)\backslash U(\AA_K)}\frac{1}{(2\pi i)^{n(r_1+2r_2)}}\sum_{\eta^*\in \Lambda^*}\int_{t\in\sigma_0+i\RR^n}F(\eta^*,t)\C(-\lambda^*,-t)(wua)\diff t\diff u\C(\lambda^*,s)(a)\diff^* a\\
&=\int_{a\in D}\frac{1}{(2\pi i)^{n(r_1+2r_2)}}\sum_{\eta^*\in \Lambda^*}\int_{t\in\sigma_0+i\RR^n}\int_{u\in U_w(\AA_K)\backslash U(\AA_K)}F(\eta^*,t)\C(-\lambda^*,-t)(wua)\diff u\diff t\C(\lambda^*,s)(a)\diff^* a\\
&=\int_{a\in D}\frac{1}{(2\pi i)^{n(r_1+2r_2)}}\sum_{\eta^*\in \Lambda^*}\int_{t\in\sigma_0+i\RR^n}F(\eta^*,t)M_w(\C(-\lambda^*,-t))(a)\C(\lambda^*,s)(a)\diff^* a\\
&=\int_{a\in D}\frac{1}{(2\pi i)^{n(r_1+2r_2)}}\sum_{\eta^*\in \Lambda^*}\int_{t\in\sigma_0+i\RR^n}F(\eta^*,t)\C(-w^{-1}(\lambda^*),-w^{-1}(t))(a)\Phi_{K,w^{-1}}(-\lambda^*,-t)\diff t \C(\lambda^*,s)(a)\diff^* a\\
&=\int_{a\in D}\mathcal{G} (F\cdot (\Phi_{K,w^{-1}}\circ(-1)))(w(a))\C(\lambda^*,s)(a)\diff^* a\\
&=\int_{a\in D}\mathcal{G} (F\cdot (\Phi_{w^{-1}}\circ(-1)))(a)\C(w(\lambda^*),w(s))(a)\diff^* a\\
&=\F(\mathcal{G} (F\cdot (\Phi_{K,w^{-1}}\circ(-1))))(w(\lambda^*),w(s))\\
&=F(w(\lambda^*),w(s))\cdot\Phi_{K,w^{-1}}(-w(\lambda^*),-w(s))\\
&=\F(\varphi)(w(\lambda^*),w(s))\cdot\Phi_{K,w}(\lambda^*,s)\\
\end{align*}as desired.
\end{proof}


\section{Isomorphism of the Hall algebra and the shuffle algebra}
\label{sec: isomorphism}

\blank Let us show that the spherical Hall algebra of $\overline{\Spec(\OO_K)}$ and Paley-Wiener the shuffle algebra associated to $\Phi_K(s)$ are isomorphic. First, we define the constant term $\CT_n$ and its twist $\~\CT_n$. We will show important properties these constant term operators satisfy. Afterwards, we define the operator $\Ch_n=\F\circ\~\CT_n$, and we will show that $\Ch=\bigoplus_{n\geqslant0}\Ch_n$ is the desired isomorphism between $SH$ and $\SH(\Phi_K)_\PW$.

\subsection{The constant term}

Let $B_n$ be the lower triangular Borel subgroup of $\GL_n$, and let $U_n$ be the unipotent radical of $B_n$. We define the constant term $\CT_n$ as an operator $\CT_n:H_n\to C_c^\infty(\B^n)$ given by $$\CT_n(a_1,\dots,a_n)=\int_{u\in U_n(\OO_K)\backslash U_n(R)}f(u\cdot \diag(a_1,\dots,a_n))\diff u=\int_{u_\AA\in U_n(K)\backslash U_n(\AA_K)}f(u_\AA\cdot \diag(a_1,\dots,a_n))\diff u_\AA$$Here, $du$ and $du_\AA$ are the normalized Hall measures on  $U_n(R)$ and $U_n(\AA_K)$, respectively, such that $U_n(\OO_K)$ and $U_n(K)$ have volume 1.

\begin{prop}\label{CT bounded}
For any $f\in H_n$, there is some $c\in\RR_+$ such that $\Supp(\CT_n(f))$ is contained in the domain $$\left\{(a_1,\dots,a_n)\in \B^n:|a_1|_\B\leqslant c,|a_1a_2|_\B\leqslant c,\dots,|a_1\dots a_{n-1}|_\B\leqslant c,\frac{1}{c}\leqslant |a_1\dots a_n|_\B\leqslant c\right\}.$$
\end{prop}

\begin{proof}
The proof is the same as in \cite{originalpaper}.
\end{proof}

We use the Haar measures $\diff g=\frac{\prod_{i,j}^n\diff g_{ij}}{\det(g)^n}$ and $\diff^* a=\prod_{i=1}^n\frac{\diff a_i}{a_i}$ for $\GL_n(R)$ and $\B^n$, respectively. Note that the Iwasawa decomposition of $\GL_n(R)$ yields $$\GL_n(R)=U_n(R)\cdot \B^n\cdot \K_n$$so we will use the notation $g=u\cdot a\cdot k$ for an element $g$ of $\GL_n(R)$. If $\diff u$ and $\diff k$ are the Haar measures on $U_n(R)$ and $\K_n$, we have that $$\diff g=\delta(a)\diff u\cdot \diff k\cdot \diff a$$where $$\delta(a)=\delta(g)=\prod_{1\leqslant i<j\leqslant n}\frac{a_j}{a_i}$$is the Iwasawa Jacobian.

Let us define the following positive definite Hermitian scalar products for $H_n$ and $C_c^\infty(\B^n)$, respectively: $$\left<f_1,f_2\right>_H=\int_{g\in\GL_n(\OO_K)\backslash \GL_n(R)}f_1(g)\overline{f_2(g)}\diff g\text{\hspace{0.5cm}and\hspace{0.5cm} }\left<\vp_1,\vp_2\right>=\frac{1}{2^n}\int_{a\in \B^n}\vp_1(a)\overline{\vp_2(a)}\diff^*a$$

We define the \textit{twisted constant term} $\~\CT_n$ as $$\~\CT_n(f)(a_1,\dots,a_n) = \CT_n(f)(a_1,\dots,a_n)\cdot\delta(a)^{\frac{1}{2}}$$

\begin{prop}
\label{prop: adjoint1}
The map $\~\CT_n:H_n\to C^\infty(\B^n)$ is adjoint to $*_{1^n}:H_1^{\otimes n}\to H_n$, that is, we have $$\left<*_{1^n}(\vp),f\right>_H=\left<\vp,\~\CT_n(f)\right>$$
\end{prop}

\begin{proof}
By Proposition~\ref{prop: * version}, we have that \begin{align*}
\left<*_{1^n}(\vp),f\right>_H &= \int_{g\in\GL_n(\OO_K)\backslash \GL_n(R)}\overline{f(g)}\sum_{\gamma\in B_n(\OO_K)\backslash \GL_n(\OO_K)}\~\vp(\gamma g)\diff g\\
&= \int_{x\in B_n(\OO_K)\backslash \GL_n(R)}\overline{f(x)}\~\vp(x)\diff x\\
&= \int_{x\in B_n(\OO_K)\backslash \GL_n(R)}\overline{f(x)}\vp(x)\delta(x)^{-\frac{1}{2}}\diff x\\
&= \frac{1}{2^n}\int_{y\in U_n(\OO_K)\backslash \GL_n(R)}\overline{f(y)}\vp(y)\delta(y)^{-\frac{1}{2}}\diff y\\
&= \frac{1}{2^n}\int _{z\in U(R)\backslash \GL_n(R)}\int_{u\in U_n(\OO_K)\backslash U_n(R)}\overline{f(uz)}\vp(z)\delta(z)^{-\frac{1}{2}}\diff u\diff z\\
&= \frac{1}{2^n}\int _{a\in\B^n}\overline{\CT_n(f)(a)}\vp(a)\delta(a)^{\frac{1}{2}}\diff^*a\\
&= \left<\varphi,\~\CT_n(f)\right>\\
\end{align*}as desired.
\end{proof}

\begin{prop}
\label{prop: adjoint2}
The map $\~\CT_n:SH_n\to C^\infty(\B^n)$ is injective.
\end{prop}

\begin{proof}
By definition of $SH$, an element $f\in SH_n$ has the form $f = *_{1^n}(\varphi)$ for some $\varphi\in H_1^{\otimes n}\subset C_c^\infty(\B^n)$. Now, assume that $f\neq0$, which implies $\varphi \neq 0$. We will show that $\~\CT_n(f)\neq 0$, from which injectivity will follow. Indeed, note that, as by Proposition~\ref{prop: adjoint1} we have that $$\left<\varphi,\~\CT_n(f)\right>=\left<\varphi,\~\CT_n(*_{1^n}(\varphi))\right>=\left<*_{1^n}(\varphi),*_{1^n}(\varphi)\right>=\left<f,f\right>_H>0.$$Then, in particular, $\~\CT_n(f)\neq 0$, as desired.
\end{proof}

\begin{prop}
\label{prop: twist and intertwiner}
We have that $$\~\CT_{n'+n''}(f'*f'')=\sum_{w\in \Sh(n',n'')}M_{w^{-1}}(\~\CT_{n'}(f')\otimes\~\CT_{n''}(f''))$$
\end{prop}

\begin{proof}
Recall that the Grassmannian $\Gr_{n'}(K^n)=P_{n',n''}(K)\backslash \GL_n(K)$ splits under the right $U(K)$-action, into $\binom{n}{n'}$ orbits (the Schubert cells) $$\Sigma_w=P_{n',n''}(K)\backslash w^{-1}U(K)$$where $w\in\Sh(n',n'').$ Now, by this and Proposition~\ref{prop: group version}, notice that

\begin{align*}
\~\CT_n(f'*f'')&=\int_{u\in U(K)\backslash U(\AA_K)}\sum_{\gamma\in B_{n',n''}(K)\backslash \GL_n(K)}f(\gamma u g)\delta(g)^{\frac{1}{2}}\diff u\\
&=\sum_{w\in\Sh(n',n'')}\int_{u\in U(K)\backslash U(\AA_K)}\sum_{v\in U_{w^{-1}}(K)\backslash U(K)}f(w^{-1}vug)\delta(g)^{\frac{1}{2}}\diff u\\
&=\sum_{w\in Sh(n',n'')}M_{w^{-1}}(\~\CT_{n'}(f')\otimes\~\CT_{n''}(f''))\\
\end{align*}as desired. 
\end{proof}

\begin{prop}
\label{prop: similar}
Let $\vp_1,\dots,\vp_n\in C_c^\infty(\B)$ and $\vp\in \vp_1\otimes\dots\otimes\vp_n\in C_c^\infty(\B^n)$. Then $$\~\CT_n(*_{1^n}(\vp))=\sum_{w\in\Perm_n}M_w(\vp).$$
\end{prop}

\begin{proof}
The proof is similar to the previous one. Note that, by Proposition~\ref{prop: * version},

\begin{align*}
\~\CT_n(*_{1^n}(\vp))&=\int_{u\in U(K)\backslash U(\AA_K)}\sum_{\gamma\in B_{n}(\OO_K)\backslash \GL_n(\OO_K)}\~\vp(\gamma u g)\delta(g)^{\frac{1}{2}}\diff u\\
&=\sum_{w\in\Perm_n}M_w(\varphi)\\
\end{align*}as desired.
\end{proof}

\subsection{The $\Ch$ homomorphism}

\blank We define the operator $\Ch_n:H_1^{\otimes n}\to \PW((\Lambda^*\times \CC)^n)$ as $\Ch_n(f)=\F(\~\CT_n(f))$ for $n\geqslant1$ and $\Ch_0:\CC\to \CC$ as $\Ch_0(c)=c$ for $n=0$. Let $\Ch:SH\to \SH(\Phi_K)_\PW$ be defined by $\Ch=\bigoplus_{n\geqslant0}\Ch_n$.

\begin{prop}
The map $\Ch$ is well-defined as a map of sets.
\end{prop}

\begin{proof}
It suffices to show that $\Ch_n(f)$ is well-defined for any $f$. We show that $\Ch_n(f)$ is well-defined and meromorphic in $(\Lambda^*)^n\times\CC_{>0}$, and then we can extend it to a meromorphic function in all of $(\Lambda^*\times\CC)^n$. The argument is the same as in \cite{originalpaper}. Finally, as $\F(H_1)=\PW(\Lambda^*\times\CC)$, $\F(SH)\subset\SH(\Phi_K)_\PW$, so $\Ch$ is well-defined, as desired.
\end{proof}Now that we have defined the map $\Ch:SH\to \bigoplus_{n\geqslant0}\PW((\Lambda^*\times\CC)^n)$, we claim that this map is the desired isomorphism $SH\overset{\sim}{\to}\SH(\Phi_K)_{\PW}$.

\begin{prop}
\label{prop: injectivity}
The map $\Ch(f)=0$ if and only if $f=0$.
\end{prop}

\begin{proof}
Clearly $\Ch(0)=0$. Let us show the opposite. If $\Ch(f)=0$, we have that $\Ch_n(f)=0$ for all $n$. This implies that $\F(\~\CT_n(f))=0$, which becomes $\~\CT_n(f)=0$ after applying $\G$, but this implies $f=0$ by Proposition~\ref{prop: adjoint2}, so the result follows.
\end{proof}

\begin{prop}
\label{prop: surjectivity}
The map $\Ch_1:SH_1=H_1\to \PW(\Lambda^*\times\CC)$ is an isomorphism.
\end{prop}

\begin{proof}
Note that $\~\CT_1(f)(a)=f(a)$, so $\Ch_1(f)=\F(f)$. However, recall that by Proposition~\ref{prop: fourier is iso}, $\F$ is an isomorphism, so $\Ch_1$ is also an isomorphism, as desired.
\end{proof}

Now, assume that $\Ch$ was an algebra homomorphism. In that case, by Proposition~\ref{prop: injectivity}, $\Ch$ would be injective.
Furthermore, by Proposition~\ref{prop: surjectivity}, $$\Ch(SH)=\bigoplus_{n\geqslant0}\Ch(H_1^{\otimes n})=\bigoplus_{n\geqslant0}\Ch(H_1)^{\otimes n}=\bigoplus_{n\geqslant0}\PW(\Lambda^*\times\CC)^{\otimes n}=\SH(\Phi_K)_{\PW}$$so $\Ch$ would also be surjective.
Thus, $\Ch$ would be the desired isomorphism.
Let us show that $\Ch$ is an algebra homomorphism.
To do this, we will introduce two lemmas.
We use $\Sh(m, n)\inv$ to denote the set $\br{w \in \Perm_{m + n} : w\inv \in \Sh(m, n)}$.

\begin{lemma} \label{lem: block * shuffle inv is symmetric group}
    We have a bijection between sets
    \begin{align*}
        \Perm_m \times \Perm_n \times \Sh(m, n)\inv & \iso \Perm_{m + n} \\
        (w', w'', w) & \mapsto (w' \times w'') w,
    \end{align*}
    where $w' \times w''$ is the permutation that permutes the first $m$ elements using $w'$ and the last $n$ elements using $w''$.
\end{lemma}

\begin{proof}
    It suffices to show that the map
    \begin{align*}
        \vp : \Sh(m, n) \times \Perm_m \times \Perm_n & \to \Perm_{m + n} \\
        (w, w', w'') & \mapsto w (w' \times w'')
    \end{align*}
    is bijective, since then we arrive at the desired conclusion by the following commutative diagram:
    \[\begin{tikzcd}
    	{\Perm_m \times \Perm_n \times \Sh(m, n)\inv} & {\Perm_{m + n}} \\
    	{\Sh(m, n) \times \Perm_m \times \Perm_n} & {\Perm_{m + n}}
    	\arrow[from=1-1, to=1-2]
    	\arrow["{\cdot^{-1}}", "\sim"', from=1-1, to=2-1]
    	\arrow["{\cdot^{-1}}", "\sim"', from=1-2, to=2-2]
    	\arrow["\vp", from=2-1, to=2-2]
    \end{tikzcd}\]

    In order to show that $\vp$ is bijective, we will first show it is surjective and deduce from the fact that both sides have the same cardinality $(m + n)!$.
    To see that $\vp$ is surjective, take any $\sigma \in \Perm_{m + n}$.
    Then there exists $\sigma_1 \times \sigma_2 \in \Perm_m \times \Perm_n$ such that $\sigma (\sigma_1 \times \sigma_2)(i) < w (\sigma_1 \times \sigma_2)(j)$ for $(i, j) \in \db{1, m}^2 \sqcup \db{m + 1, m + n}^2$, i.e., $\sigma (\sigma_1 \times \sigma_2) \in \Sh(m, n)$, so we see $\sigma = \vp(\sigma (\sigma_1 \times \sigma_2), \sigma_1\inv, \sigma_2\inv) \in \im \vp$.

    Now note that the domain of $\vp$ has cardinality $m! \cdot n! \cdot \binom{m + n}{m} = (m + n)!$, while the codomain of $\vp$ also has cardinality $(m + n)!$, so we see $\vp$ is injective and finishes our proof.
\end{proof}

\begin{lemma} \label{lem: block * shuffle inv has max length}
    For $w' \in \Perm_m, w'' \in \Perm_n, w \in \Sh(m, n)\inv$, we have
    \[
    \#\I_{(w' \times w'')w} = \#\I_{w' \times w''} + \#\I_w.
    \]
    
    More precisely, we have
    \[
    \I_{(w' \times w'')w} = \pr{w\inv \I_{w' \times w''}} \sqcup \I_w.
    \]
\end{lemma}

\begin{proof}
    The two maps
    \begin{align*}
        \I_{w' \times w''} & \to \I_{(w' \times w'')w} \\
        (i, j) & \mapsto (w\inv(i), w\inv(j)) \\
        \I_w & \to \I_{(w' \times w'')w} \\
        (i, j) & \mapsto (i, j)
    \end{align*}
    are injective, and the image of these two maps do not have intersection because otherwise there exist $i < j$ such that $(w\inv(i), w\inv(j)) \in \I_w$, which implies $i > j$, a contradiction.

    As a result, we see $\#\I_{(w' \times w'')w} \ge \#\I_{w' \times w''} + \#\I_w$.
    However, for $w_1, w_2 \in \Perm_{m + n}$, we have $\#\I_{w_1 w_2} \le \#\I_{w_1} + \#\I_{w_2}$, which shows $\#\I_{(w' \times w'')w} = \#\I_{w' \times w''} + \#\I_w$, as desired.
\end{proof}

\begin{cor} \label{cor: block and shuffle inv are contravariant in M}
    For $w' \times w'' \in \Perm_m \times \Perm_n, w \in \Sh(m, n)\inv$, we have
    \[
    M_w M_{w' \times w''} = M_{(w' \times w'') w}.
    \]
\end{cor}

\begin{proof}
    This follows from Corollary~\ref{cor: M contravariant when everything is good} and Lemma~\ref{lem: block * shuffle inv has max length}.
\end{proof}

As $\F$ and $\~\CT_n$ are linear, $\Ch(f'+f'')=\Ch(f')+\Ch(f'')$ follows.
We also have that $\Ch(1)=1$.
It remains to show that  $\Ch_n(f'*f'')=\Ch_{n'}(f')\shufflemult \Ch_{n''}(f'')$ for $f'\in H_1^{\otimes n'}$ and $f''\in H_1^{\otimes n''}$.
Let us do this.
First, we compute $\Ch_n(f'*f'')(\lambda^*,s)$.
If $f'=*_{1^{n'}}(\varphi')$ and $f''=*_{1^{n''}}(\varphi'')$, by Proposition~\ref{prop: twist and intertwiner}, Proposition~\ref{prop: similar}, Corollary~\ref{cor: block and shuffle inv are contravariant in M}, Lemma~\ref{lem: block * shuffle inv is symmetric group}, and Proposition~\ref{prop: 5.21},

\begin{align*}
\Ch_n(f'*f'')(\lambda^*,s)&=\F(\widetilde{\CT}_{n'+n''}(f'*f''))(\lambda^*,s)\\
&=\F\left(\sum_{w\in \Sh(n',n'')}M_{w^{-1}}(\widetilde{\CT}_{n'}(f')\otimes\widetilde{\CT}_{n''}(f''))\right)(\lambda^*,s)\\
&=\F\left(\sum_{w\in \Sh(n',n'')}M_{w^{-1}}\left(\sum_{w'\in\Perm_{n'}}M_{w'}(\varphi')\otimes \sum_{w''\in\Perm_{n''}}M_{w''}(\varphi'')\right)\right)(\lambda^*,s)\\
&=\F\left(\sum_{w\in \Sh(n',n'')}\sum_{w'\in\Perm_{n'}}\sum_{w''\in\Perm_{n''}}M_{w^{-1}}(M_{w'}(\varphi'))\otimes M_{w^{-1}}(M_{w''}(\varphi''))\right)(\lambda^*,s)\\
&=\F\left(\sum_{w\in \Sh(n',n'')}\sum_{\substack{w'\in\Perm_{n'}\\w''\in\Perm_{n''}}}M_{(w'\times w'')w^{-1}}(\vp)\right)(\lambda^*,s)\\
&=\F\left(\sum_{\sigma\in\Perm_n}M_{\sigma}(\vp)\right)(\lambda^*,s)\\
&=\sum_{\sigma\in\Perm_n}\F(\vp)(\sigma(\lambda^*),\sigma(s))\cdot\Phi_{K,\sigma}(\lambda^*,s)\\
\end{align*}

Now, we compute $(\Ch_n(f)\shufflemult \Ch_{n''}(f''))(\lambda^*,s)$. If $f'=*_{1^{n'}}(\varphi')$ and $f''=*_{1^{n''}}(\varphi'')$, by definition of shuffle algebra, Proposition~\ref{prop: similar},  Proposition~\ref{prop: 5.21}, Corollary~\ref{cor: block and shuffle inv are contravariant in M}, and Lemma~\ref{lem: block * shuffle inv is symmetric group}, 

\begin{align*}
(\Ch_{n'}(f')\shufflemult \Ch_{n''}(f''))(\lambda^*,s)&=\sum_{w\in \Sh(w,w')}(\Ch_{n'}(f')\otimes \Ch_{n''}(f''))(w^{-1}(\lambda^*),w^{-1}(s))\cdot \Phi_{K,w^{-1}}(\lambda^*,s)\\
&=\sum_{w\in \Sh(w,w')}(\F(\widetilde{\CT}_{n'}(f'))\otimes \F(\widetilde{\CT}_{n''}(f'')))(w^{-1}(\lambda^*),w^{-1}(s))\cdot\Phi_{K,w^{-1}}(\lambda^*,s)\\
&=\sum_{w\in \Sh(w,w')}\left(\F\left(\sum_{w'\in\Perm_{n'}}M_{w'}(\varphi')\right)\F\left(\sum_{w''\in\Perm_{n''}}M_{w''}(\varphi'')\right)\right)(w^{-1}(\lambda^*),w^{-1}(s))\cdot \Phi_{K,w^{-1}}(\lambda^*,s)\\
&=\sum_{w\in \Sh(w,w')}\Phi_{K,w^{-1}}(\lambda^*,s)\\
&\hspace{0.5cm}\cdot\sum_{\substack{w'\in\Perm_{n'}\\w''\in\Perm_{n''}}}\F(M_{w'}(\varphi'))(w^{-1}({\lambda^*}' ),w^{-1}(s'))\F(M_{w''}(\varphi''))(w_{-1}({\lambda^*}''),w^{-1}(s''))\\
&=\sum_{w\in \Sh(w,w')}\Phi_{K,w^{-1}}(\lambda^*,s)\\
&\hspace{0.5cm}\cdot\sum_{\substack{w'\in\Perm_{n'}\\w''\in\Perm_{n''}}}\F(\varphi')(w'w^{-1}({\lambda^*}'),w'w^{-1}(s'))\F(\varphi'')(w''w^{-1}({\lambda^*}''),w''w^{-1}(s''))\\
&\hspace{0.5cm}\cdot\Phi_{K,w'}(w^{-1}({\lambda^*}'),w^{-1}(s'))\Phi_{K,w''}(w^{-1}({\lambda^*}''),w^{-1}(s''))\\
&=\sum_{w\in \Sh(w,w')}\Phi_{K,w^{-1}}(\lambda^*,s)\Phi_{K,w'\times w''}(w^{-1}(\lambda^*),w^{-1}(s))\\
&\hspace{0.5cm}\cdot\sum_{\substack{w'\in\Perm_{n'}\\w''\in\Perm_{n''}}}\F(\varphi)((w'\times w'')w^{-1}(\lambda^*),(w'\times w'')w^{-1}(s))\cdot\\
&=\sum_{w\in \Sh(w,w')}\sum_{\substack{w'\in\Perm_{n'}\\w''\in\Perm_{n''}}}\F(\varphi)((w'\times w'')w^{-1}(\lambda^*),(w'\times w'')w^{-1}(s))\cdot\Phi_{K,(w'\times w'')w^{-1}}(\lambda^*,s)\\
&=\sum_{\sigma\in \Perm_n}\F(\varphi)(\sigma(\lambda^*),\sigma(s))\cdot\Phi_{K,\sigma}(\lambda^*,s)\\
\end{align*}which agrees with our computation for $\Ch_n(f'*f'')(\lambda^*,s)$. Hence, $\Ch:SH\to\SH(\Phi_K)_\PW$ is an isomorphism.

\bibliographystyle{alpha}
\bibliography{HallAlgebra}


\end{document}